\documentclass[12pt]{article}
\usepackage[
    a4paper,      
    top=3.5cm,     
    bottom=3.5cm, 
    left=3.5cm,     
    right=3.5cm   
]{geometry}
\usepackage{amsthm,amsmath,amssymb}
\usepackage{mathrsfs}
\usepackage[all]{xy}
\usepackage{amsmath}
\usepackage{tikz}
\usepackage{tikz-cd}
\usepackage{array}
\usepackage{booktabs}
\numberwithin{equation}{section}
\newtheorem{definition}{Definition}[section]
\newtheorem{theorem}{Theorem}[section]
\newtheorem{conjecture}{Conjecture}
\newtheorem*{remark}{Remark}
\newtheorem{corollary}{Corollary}[section]
\newtheorem{proposition}{Proposition}[section]
\newtheorem{lemma}{Lemma}[section]

\title{Semi-flat constant scalar curvature K\"ahler metric on elliptic surface}
\author{Zhenqu Wang, Zhenlei Zhang}

\begin{document}

\maketitle
\date{}

\begin{abstract}
  We introduce and construct a novel type of canonical metric: the semi-flat constant scalar curvature K\"hler (semi-flat cscK) current, which naturally arises in Calabi-Yau fibrations. For a given elliptic surface X with a holomorphic section, We explicitly construct the desired semi-flat cscK current and analyze its behavior along singular parts. We establish its uniqueness under the condition that X possesses at least one singular fiber other than of type $I_b$ or $I_b^*$. These results contribute to a geometric uniformization program for elliptic surfaces.
\end{abstract}
\tableofcontents
\setcounter{section}{-1}
\section{Introduction}
\ \ \ \ The search for canonical metrics on compact K\"ahler manifolds $X$ is a central problem in complex geometry. In this paper, we assume that $X$ is
minimal  with semi-ample canonical bundle $K_X$ and positive Kodaira dimension, conditions under which the canonical class (denoted by $K_X$ as well) lies on the boundary of the K\"ahler cone. We then consider the canonical metrics related to $K_X$.

There are several constructions to the canonical metrics in this case. Firstly, for projective $X$, Song and Tian \cite{SongTian06} introduced the generalized K\"ahler-Einstein current $\omega_{can}$ on the canonical model $X_{can}$ of $X$. This current satisfies the condition $f^*\omega_{can}\in 2\pi K_X$, where $f:X\rightarrow X_{can}$ is the pluricanonical map. They established its existence, first for the case of Kodaira dimension 1 in \cite{SongTian06}, and later extended this result to all positive Kodaira dimension in \cite{SoTi12}. This result extends the known existence of (possibly singular) K\"ahler Einstein metrics for varieties of general type, which was established earlier for case $K_X>0$ by Aubin \cite{Aubin76}, Yau \cite{Yau78}, and more general case by Tsuji \cite{Tsuji88}, 
Kobayashi \cite{Ko85}, Tian-Zhang \cite{TiZh06}, Eyssidieux-Guedj-Zeriahi \cite{EyViZe09} etal. The notion of generalized K\"ahler-Einstein metric is defined as follows:
\begin{definition}[Song-Tian \cite{SongTian06}]
Let $f=|K_X^m|:X\rightarrow X_{can}\subset \mathbb{P}^{N_m},N_m=\dim H^0(X,H_X^m)$ be the pluri-canonical map of $X$ to its canonical model
$X_{can}$. Let $\omega$ be a possibly singular K\"ahler metric on $X_{can}$ such that $f^*\omega\in -2\pi c_1(X)$. Then $\omega$ is called a generalized K\"ahler-Einstein metric if on $X_{can}^*$:
\begin{equation*}
  Ric(\omega)=-\omega+\omega_{WP}
\end{equation*}
where $X_{can}^*$ is the set of all points $s\in X_{can}$ with $X_s=f^{-1}(s)$ being nonsingular, and $\omega_{WP}$ is the Weil-Petersson metric on the base.

In general, given any Calabi-Yau fibration: $f:X\rightarrow \Sigma$ over an algebraic variety $\Sigma$. If $X$ is nonsingular and
$\Sigma^*$ is the set of all points $s\in \Sigma$ with $X_s=f^{-1}(s)$ being nonsingular. Then a possibly singular K\"ahler metric on $\Sigma$
is called a generalized K\"ahler-Einstein metric if on $\Sigma^*$:
\begin{equation*}
  Ric(\omega)=\lambda\omega+\omega_{WP},\lambda=-1,0,1
\end{equation*}
\end{definition}
Under this definition, the (possibly singular) K\"ahler Einstein metric on a minimal algebraic manifold of general type mentioned above can be regarded as a specific instance of a generalized K\"ahler Einstein metric according to the first part of this definition.

The generalized K\"ahler-Einstein metric can be elegantly obtained by using the (normalized) K\"ahler-Ricci flow equation, first introduced by Hamilton \cite{Hamilton82} and adapted to the K\"ahler setting by Cao \cite{Cao86}.
\begin{equation}\label{Smooth KRF}
\begin{split}
\left\{
\begin{aligned}
  &\frac{\partial \omega(t)}{\partial t}=-Ric(\omega(t))-\omega(t)\\
  &\omega(t)=\omega_0
\end{aligned}
\right.
\end{split}
\end{equation}
Song-Tian established the long-time existence and convergence of this equation when $K_X$ is semi-ample and positive Kodaira dimension:

\begin{theorem}[Song-Tian \cite{SoTi12}]
Let X be a nonsingular algebraic variety with semi-ample canonical line bundle $K_X$ and so $X$ admits an algebraic fibration $f:X\rightarrow X_{can}$
over its canonical model $X_{can}$. Suppose $0<\dim X_{can}=\kappa<\dim X=n$. Then for any initial K\"ahler metric, the K\"ahler Ricci flow \ref{Smooth KRF} has a global solution $\omega(t,\cdot)$ for all time $t\in [0,\infty)$ satisfying:
\begin{itemize}
\item[1.] $\omega(t, \cdot)$ converges to $f^*\omega_{can}\in -2\pi c_1(X)$ as currents for a positive closed $(1,1)$-current $\omega_{can}$ on $X_{can}$ with continuous potential.
\item[2.] $\omega_{can}$ is smooth on $X_{can}^{\circ}$ and satisfies the generalized K\"ahler-Einstein equation on $X_{can}^{\circ}$:
\begin{equation*}
  Ric(\omega_{can})=-\omega_{can}+\omega_{WP}.
\end{equation*}
\end{itemize}
\end{theorem}

Song-Tian \cite{SongTian06} also conjectured the convergence in the Gromov-Hausdorff topology to $(X_{can},\omega_{can})$. Progress towards this conjecture can be found in \cite{Gross-Tosatti-Zhang20,GuoSongWe16,JianSong22,SongTian12,SongTianZhang19,ZhZh19}. For more on the K\"ahler-Ricci flow (\ref{Smooth KRF}) in the smooth minimal model case, especially regarding its long-time behavior, the structure of its singularities, and collapsing phenomena, see \cite{Cao86, Fong-FT-ZY20, Fong-FT-ZZ15, GuoSongWe16, JianSong22, SongTian06, SongTian12, SongTian16, Song-Tian17, SongTianZhang19, SoWe11, TiZh06, Tian-Zhang07, TianZLZhang16, TianZLZhang21, Tosatti10, TWY18, Tsuji88, Wang17, Z.Zh09, Z.Zh10}.

On the another hand, from the view point of constant scalar curvature K\"ahler (cscK) metrics, by the result of Song-Shi-Jian \cite{JianShiSong18} there always exist constant scalar curvature K\"ahler (cscK) metrics $\omega_{\epsilon}$ in the class $K_X+ \epsilon [\omega]$ for any given K\"ahler class $[\omega]$ and sufficiently small $\epsilon$, a candidate of "canonical metric" around the canonical class $K_X$ (the existence of csck is valid even for nef $K_X$ by Dyrefelt \cite{Dy22} and Song \cite{Song20}). It is also conjectured by Song-Shi-Jian \cite{JianShiSong18} that the cscK metrics in perturbed class converge in the Gromov-Hausdorff sense as well to $(X_{can},\omega_{can})$. Recent progress to the conjecture consists of Liu \cite{Liu20} and Guo-Jian-Shi-Song \cite{GuoJianShiSong24}. See \cite{Fine04,Dyrefelt20,Zh25} for more related works.

Finally, since the pluricanonical map $f:X\rightarrow X_{can}$ in our case is a Calabi-Yau fibration which mean the regular fibers $f^{-1}(s),s\in X_{can}$ possess trivial first Chern class. Semi-flat metrics arise naturally in the context of a Calabi-Yau fibration $f:X\rightarrow \Sigma$. On an open subset $U \subset \Sigma$ where the fibers $f^{-1}(s)$ for $s \in U$ are nonsingular, a semi-flat metric on $f^{-1}(U)$ is defined as a K\"ahler metric whose restriction to each fiber is Ricci-flat.

In this paper, we focus on searching for the "canonical metric" in the field of semi-flat metrics. Let us start with some frameworks: let $\Phi: X \to \Sigma_g$ be a minimal elliptic surface with a holomorphic section $s:\Sigma_g\to X$. Note that the existence of a section is equivalent to that each fibre has multiplicity $1$. Following Kodaira's work \cite{Ko63} the elliptic surface $(X,\Phi,s)$ can be reconstructed via a functional invariant $J: \Sigma_g \to \mathbb{P}^1$ and related  elliptic monodromy $\rho: \pi_1(\Sigma_g^*) \to \mathrm{SL}_2(\mathbb{Z})$, where $\Sigma_g^*$ is $\Sigma_g$ punctured at finitely many points. Choose a multi-valued period map $\omega: \Sigma_g^* \to \mathbb{H}^+$ such that $ J =j \circ \omega $ for $j$ the elliptic modular function, a classic object in the theory of modular forms \cite{Zagier08}. The analytic continuation of $\omega$ yields a monodromy representation $\rho': \pi_1(\Sigma_g^*) \to \mathrm{PSL}_2(\mathbb{Z})$, then $\rho: \pi_1(\Sigma_g^*) \to \mathrm{SL}_2(\mathbb{Z})$ is a lift of $\rho'$. The desired $(X,\Phi,s)$ with prescribed invariants $(J,\rho)$ can be constructed uniquely \cite{Ko63}.

Specifically, for an open subset $B \subset \Sigma_g$ where the fibers $\Phi^{-1}(s)$ are nonsingular for all $s \in B$, there exists a standard semi-flat, Ricci-flat K\"ahler metric on $\Phi^{-1}(B)$ \cite{GrWi20}. Such a metric is also known as a semi-flat Calabi-Yau metric. It is naturally to ask for a global canonical semi-flat metric over the elliptic surface $X$. In fact what we are going to seek for is a more robust structure:

\begin{definition}
    A semi-flat constant scalar curvature K\"ahler (semi-flat cscK) current $\omega$ on a minimal elliptic surface $\Phi:X\rightarrow \Sigma_g$ is a closed positive $(1,1)$-current on $X$, whose restriction to $X^*:=X\setminus \Phi^{-1}(\mathcal{P})$ is a semi-flat cscK metric, where $\mathcal{P}\subset \Sigma_g$ is a finite set.
\end{definition}

Let $H(X,\Phi,s)$ denote the group of automorphisms that preserve the fibration structure and the holomorphic section:
\begin{equation*}
    H(X,\Phi,s)=\{(h',h)\in Aut(X)\times Aut(\Sigma_g)| \Phi\circ h'=h\circ \Phi, h'\circ s=s\circ h\}
\end{equation*}
The main results of this paper can be stated as follows:

\begin{theorem}\label{T1}
Let $(X,\Phi,s)$ be an elliptic surface with a section. There exists an $H(X,\Phi,s)$-invariant semi-flat cscK current $\eta_X$ on $X$ with regular fiber volume $1$ and scalar curvature $-3$. Furthermore, if $X$
possesses at least one singular fiber $X_p$ other than of type $I_b,I_b^*$,
then this current is unique.
\end{theorem}

Let $d$ be the degree of the Jacobian $J:\Sigma_g \rightarrow \mathbb{P}^1$, and $\chi=h^0(X,\mathcal{O}_X)-h^1(X,\mathcal{O}_X)+h^2(X,\mathcal{O}_X)$ be the Euler-Poinca\'re character of $X$. We have:
\begin{theorem}\label{T2}
    If the genus $g\ge 1$, there is a unique cohomology class $D_X\in H^{1,1}(X)\cap H^2(X,\mathbb{Q})$, such that an $H(X,\Phi,s)$-invariant semi-flat cscK current $\eta(t)$ in the class $K_X+tD_X$ exists if and only if  $t\in \left(0,1+\frac{2g-2}{\chi}\right)$. The fiber volume is exactly $t$; the scalar curvature of $\eta(t)$ is 
\begin{equation*}
\begin{split}
\text{scalar curvature}&=\frac{-2\pi d}{2g-2+\chi(1-t)}.
\end{split}
\end{equation*}
\end{theorem}

Note that the unique metric mentioned in the Theorem \ref{T1} corresponds to $t_0^{-1}\eta (t_0)$ that described in the Theorem \ref{T2}, where:
\begin{equation*}
  t_0=\frac{2g-2+\chi}{\pi d/3+\chi}
\end{equation*}

For comparison, consider the K\"ahler Ricci flow (\ref{Smooth KRF}) on elliptic surface with initial K\"ahler metric $\omega_0$.
Let $X_{s_i}=m_iF_i$ be its singular fibers drived from its canonical model $f:X\rightarrow X_{can}$ with multiplicity $m_i$, $i=1,...,k$.
Song-Tian \cite{SongTian06} proved the long-time existence of the solution $\omega(t)$. They showed that this solution converges to $f^*\omega_{\infty}\in K_X$ and $\omega_{\infty}$
is a generalized K\"ahler-Einstein metric on $X_{can}$ satisfying:
\begin{equation*}
  Ric(\omega_{\infty})=-\omega_{\infty}+\omega_{WP}+2\pi \sum\limits_{i=1}^{k}\frac{m_i-1}{m_i}[s_i]
\end{equation*}
where $[p]$ is the current of integration associated to the divisor $p$ on $\Sigma_g$.

The semi-flat cscK $\omega(t)$ satisfies a special twisted K\"ahler-Ricci flow. For any initial semi-flat cscK current $\omega_0=\eta_X(\delta,\epsilon)=h_{\tilde{\omega}}^*\eta(\delta,\epsilon)$ for a fixed $\epsilon$ (see equation (\ref{eta delta epsilon})), we have a series of semi-flat cscK $\omega(t)$
\begin{equation}\label{SKRF solution}
  \omega(t)=\eta_X((\delta-3/2)e^{-t}+3/2,\epsilon e^{-t}).
\end{equation}
In the language of currents $\eta(t)$ in the Theorem \ref{T2} we have $\omega(t)=a(t)^{-1}\eta(\epsilon e^{-t}a(t))$, where
\begin{equation*}
  a(t)=\frac{2g-2+\chi}{\frac{\pi d}{3}\left((\delta-\frac{3}{2})e^{-t}+\frac{3}{2}\right)+\epsilon e^{-t}\chi}.
\end{equation*}
Then $\omega(t)$ satisfies 
\begin{equation}\label{Singular KRF}
\begin{split}
\left\{
\begin{aligned}
  &\frac{\partial \omega(t)}{\partial t}=-Ric(\omega(t))-\omega(t)-2\pi \sum_{p \in \mathcal{P}} \left(\frac{d_p}{\mu_p}-1+\delta_p\right) [X_p],\\
  &\omega(t)=\omega_0,
\end{aligned}
\right.
\end{split}
\end{equation}
where $[X_p]$ denotes the current of integration along the divisor $X_p = \Phi^{-1}(p)$. The local invariants $d_p=ord_pJ$ and $\mu_p$ is 
\begin{equation*}
\mu_p=\left\{
\begin{aligned}
1, & \ J(p)\ne 0,1,\infty,\\
3, & \ J(p)=0,\\
2, & \ J(p)=1,\\
\infty, &\ J(p)=\infty,
\end{aligned}
\right.
\end{equation*} 
and $\delta_p$ is listed in the table 
\begin{table}[htbp]
  \centering
  \renewcommand{\arraystretch}{1.3}
  \caption{List of $\delta_p$}
  \label{table delta p}
  \begin{tabular}{cccccccccc}
    \toprule
 Fiber types
    & Regular & $\mathrm{I}_b,b\ge 1$ & $\mathrm{I}_b^*,b\ge 0$ & $\mathrm{II}$ & $\mathrm{II}^*$ & $\mathrm{III}$ & $\mathrm{III}^*$ & $\mathrm{IV}$ & $\mathrm{IV}^*$\\
    \midrule
    $\delta_p$ & $0$ & $0$ & $\frac{1}{2}$ & $\frac{1}{6}$ & $\frac{5}{6}$ & $\frac{1}{4}$ & $\frac{3}{4}$ & $\frac{1}{3}$ & $\frac{2}{3}$\\
    \bottomrule
  \end{tabular}
  \renewcommand{\arraystretch}{1}
\end{table}

We discuss the limit of the current $\omega(t)$. Let $\eta_{-1}$ be the Poincar\'e metric on $\mathbb{H}^+$, this metric descends via the elliptic modular funtion
$j:\mathbb{H}^+\cup \{\infty\}\rightarrow \mathbb{P}^1$ to a singular K\"ahler metric $\eta_{-1}'$ on $\mathbb{P}^1$ with constant curvature $-1$,
let $\eta_J:=J^*\eta_{-1}'$. Then as $t\rightarrow 0^+$, the solution (\ref{SKRF solution}) has a limit:
\begin{equation*}
\lim\limits_{t\rightarrow \infty}\omega(t)=\Phi^*\omega_{\infty},\qquad
 \omega_{\infty}=\frac{3}{2}\eta_J.
\end{equation*}
We note that the Weil-Peterssion metric is $\omega_{WP}=\frac{1}{2}\eta_J$. The limit satisfies a similar twisted K\"ahler-Einstein equation
\begin{equation}\label{cscK limits}
\begin{split}
  Ric(\omega_{\infty}) &=-\omega_{\infty}+\omega_{WP}-2\pi \sum_{p \in \mathcal{P}} \left(\frac{d_p}{\mu_p}-1\right) [p].
\end{split}
\end{equation}

This paper is organized as follows:

In section 1, we provide a brief review of elliptic surfaces and discuss the structure of the cohomology group $H^{1,1}(X)\cap H^2(X,\mathbb{R})$.

In section 2, we construct a certain semi-flat cscK metric on $X^*$, where
$X^*$ is the space $X$ with finitely many fibers removed. We begin with the Siegel-Jacobi space $\mathbb{H}^+\times\mathbb{C}$ equipped with the K\"ahler metric\cite{BeRo82,Yang07}:
\begin{equation*}
    \eta=\sqrt{-1}\partial\bar{\partial}\left(-\log v^2+\frac{y^2}{v}\right), (u+iv,x+iy)\in \mathbb{H}^+\times\mathbb{C}
\end{equation*}
This metric descends to a semi-flat cscK metric on the universal elliptic surface which can be realized as the quotient:
\begin{equation*}
    \mathbb{H}^+\times_{Id} \mathbb{C}:=\mathbb{H}^+\times\mathbb{C}/
    (\zeta,z)\sim(\zeta,z+n\zeta+m), n,m\in \mathbb{Z}
\end{equation*}
We prove (see Theorem \ref{universal-rigidity}) that this metric is the unique $SL_2(\mathbb{Z})$-invariant semi-flat cscK metric on $\mathbb{H}^+\times_{Id}\mathbb{C}$ with fiber volume $1$ and scalar curvature $-3$. $X^*$ is covered by an elliptic surface denoted by $\mathbb{H}^+\times_{\tilde{\omega}}\mathbb{C}$, which is constructed as the quotient:
\begin{equation*}
    \mathbb{H}^+\times_{\tilde{\omega}}\mathbb{C}
=\mathbb{H}^+\times\mathbb{C}/(\zeta,z)\sim(\zeta,z+n\tilde{\omega}(\zeta)+m),
n,m\in \mathbb{Z}
\end{equation*}
A naturally defined morphism \eqref{homega} then pulls back $\eta$ to a well-defined $\pi_1(\Sigma_g^*)$-invariant semi-flat cscK metric on $X^*$.

In section 3, we  demonstrate the $H(X,\Phi,s)$-invariant of $\eta_X$, and establish its uniqueness under the additional condition that $X$ possesses at least one singular fiber $X_p$ of a type other than $I_b$ and $I_b^*$. The crucial step is to establish a correspondence: either by lifting an element  $(h',h)\in H(X,\Phi,s)$
to an automorphism of its cover $\mathbb{H}^+\times_{\tilde{\omega}}\mathbb{C}$,
or by projecting a specific automorphism of $\mathbb{H}^+\times_{\tilde{\omega}}\mathbb{C}$
to an element in $H(X,\Phi,s)$. Furthermore, proving the uniqueness requires a strengthened version of Theorem  \ref{universal-rigidity}, namely Lemma \ref{strongversion}, and the conditions of this lemma are met thanks to the singular fiber condition.

In section 4, we extend $\eta_X$  to the entire space $X$ as a closed positive $(1,1)$-current. We show that the local potential of $\eta_X$ around each singular fiber is locally $L^1$-integrable, except at finitely many points. By applying Theorem \ref{extensionfinite}, an extension theorem due to Harvey and Polking \cite{HaPo75}, we can further extend the current $\eta_X$ across these remaining finite points.

In section 5, we prove that the associated class of $\eta_X$ is algebraic and the analysis of local behavior of $\eta_X$ along singular fibers then allows us to precisely compute the trivial lattice component of $[\eta_X]$ as detailed in Proposition \ref{class}.
Subsequently, we consider a collapsing series of semi-flat cscK currents which enables us to obtain the class $D_X\in H^{1,1}(X)\cap H^2(X,\mathbb{Q})$ introduced in Theorem \ref{series}.

Finally, in Section 6, we propose the following conjecture:
\begin{conjecture}\label{conjecture}
Let $\Phi:X\rightarrow \Sigma_g$ be a minimal elliptic surface,
    there is a semi-flat cscK current on $X$  with regular fiber volume $1$ and
    scalar curvature $-3$, which is unique in the following sense: if $\omega_1,\omega_2$ are any two such semi-flat cscK currents, then there exists an element $(h',h)\in H(X,\Phi)$, such that $h'_*\omega_1=\omega_2$. Here, $H(X,\Phi)=\{(h',h)\in Aut(X)\times Aut(\Sigma_g)| \Phi\circ h'=h\circ \Phi\}$.
\end{conjecture}

The main motivation to the paper is a version of geometrical uniformization for elliptic surface. As a corollary to this conjecture, for any given elliptic surface with a section $(X,\Phi,s)$ there associates a unique semi-flat cscK up to $H(X,\Phi)$ action; moreover, there is a class $D_X$ as in Theorem \ref{T2} associated to the metric which is also unique up to a pull back by $T_{\sigma}$ for some $\sigma\in MW(X/\Sigma_g)$,
where $MW(X/\Sigma_g)$ denotes the Mordell-Weil group of the elliptic surface $\Phi:X\rightarrow \Sigma_g$ consisting of all holomorphic sections and $T_{\sigma}\in Aut(X)$ is the automorphism of $X$ corresponding to translation along fibers by $\sigma\in MW(X/\Sigma_g)$. What we proved is a weaker version of the conjecture. 

\begin{remark}
 We are grateful to Professor Valentino Tosatti for a valuable personal communication. He kindly informed us of related work by his student Gregory Edwards (in 2017, concerning special cases of these metrics locally near singular fibers of type $I_b$) and offered the valuable clarification that while the extendability to a positive (1,1)-current is central in our semi-flat cscK case, this property is not universally true for all semi-flat form.\cite{CaoGuenanciaPaun23}.
\end{remark}


\section{A brief review of elliptic surface}
\subsection{Non-singular locus}

This section offers a brief review of elliptic surfaces, more details see Kodaira's seminal work \cite{Ko63}.

Let $U$ be a Riemann surface, $\mathbb{H}^+$ be the upper half-plane, and $\phi:U\rightarrow \mathbb{H}^+$ a non-constant holomorphic map. we define the fiber product $U\times_{\phi}\mathbb{C}$ as the quotient of $U\times\mathbb{C}$ by the equivalence relation: $(\zeta,z)\sim(\zeta,z+m\phi(\zeta)+n),m,n\in \mathbb{Z}$, This construction yields an elliptic fibration over $U$:
\begin{equation*}
\begin{split}
    U\times_{\phi}\mathbb{C} &\rightarrow U\\
    [\zeta,z] &\rightarrow \zeta
\end{split}
\end{equation*}

The elliptic fibration $\mathbb{H}^+\times_{Id}\mathbb{C}$
is known as the universal elliptic fibration. It is related to any other elliptic fibration of the form
$U\times_{Id}\mathbb{C}$ via a natural morphism $h_{\phi}$, which is defined as below:
\begin{equation}\label{homega}
    \begin{split}    h_{\phi}:U\times_{\phi}\mathbb{C}&\rightarrow \mathbb{H}^+\times_{Id}\mathbb{C}\\
        [\zeta,z]&\rightarrow [\phi(\zeta),z]
    \end{split}
\end{equation}
Let the group $\mathcal{G}$ consists of pairs $g(A,\textbf{n}),A\in SL_2(\mathbb{Z}),\textbf{n}\in \mathbb{Z}^2$ and
\begin{equation*}
\begin{split}
      g(A,\textbf{n})\cdot g(B,\textbf{m}):&=g(AB,\textbf{n}B+m)\\
    g(A,\textbf{n}):&=g(A^{-1},-\textbf{n}A^{-1})
\end{split}
\end{equation*}
let $\mathcal{N}\triangleleft \mathcal{G}$ be the normal subgroup consists of element $g(\textbf{1},\textbf{n}),\textbf{n}\in \mathbb{Z}^2$, then there is a $\mathcal{G}$-action on
$\mathbb{H}^+\times \mathbb{C}$:
\begin{equation*}
\begin{split}
\mathcal{G}\times \mathbb{H}^+\times\mathbb{C}&\rightarrow \mathbb{H}^+\times \mathbb{C}\\
    g(A,\textbf{n})\times (\zeta,z) &\rightarrow \left(\frac{a\zeta+b}{c\zeta+d},\frac{z+n_1\zeta+n_2}{c\zeta+d}
    \right)\\
    where\ A=\begin{pmatrix}
        a & b\\
        c & d
    \end{pmatrix},\ &\textbf{n}=(n_1,n_2)
\end{split}
\end{equation*}
then $\mathbb{H}^+\times_{Id}\mathbb{C}=(\mathbb{H}^+\times \mathbb{C})/\mathcal{N}$, and $\mathcal{G}/\mathcal{N}\cong Sl_2(\mathbb{Z})$, henceforth, we refer to the $SL_2(\mathbb{Z})$-action on $\mathbb{H}^+\times_{Id}\mathbb{C}$ be the induced $\mathcal{G}/\mathcal{N}$ action on $\mathbb{H}^+\times_{Id}\mathbb{C}$.

More general, let $\Phi:X\rightarrow \Sigma_g$ be an elliptic fibration based on a compact Riemann surface with genus $g$, such that $\Phi$ posses a globally defined holomorphic section $s:\Sigma_g\rightarrow X$, there is a finite set $S\subset \Sigma_g$, and a multi-valued holomorphic map $\omega$ form $\Sigma_g-S$ to the upper half plane $\mathbb{H}^+$, such that $X_p:=
\Phi^{-1}(p)$ biholomorphic to the torus with period $1$ and $\omega(p)$,  let $j$ be the elliptic modular function with $j(e^{2\pi \sqrt{-1}/3})=0,j(\sqrt{-1})=1,j(\infty)=\infty$, then $J=j\circ\omega$
is a well-defined non-constant meremorphic function on $\Sigma_g$ whcih is called the Jacobian of the elliptic fibration.
Take a finite set $\mathcal{S}=\{0,1,\infty,s_1,...,s_m\}\subset \mathbb{P}^1$ such that
$mult_pJ=1,\forall J(p)\notin \mathcal{S}$ and $X_p$ is regular $\forall p\notin \mathcal{S}$,
let $\mathcal{P}:=J^{-1}(\mathcal{S})$, $\Sigma^*_g:=\Sigma_g-\mathcal{P},  \mathbb{H}^+_*:=\mathbb{H}^+-j^{-1}(\mathcal{S}), X^*:=\Phi^{-1}(\Sigma_g^*)$. Since the Euler nnumber of $\Sigma_g^*$ is negative, the universal covering of $\Sigma_g^*$ must be the upper half plane $\mathbb{H}^+$, let $\pi:\mathbb{H}^+\rightarrow \Sigma_g^*$ be the universal covering of $\Sigma_g^*$.  Moreover, we can treat $\omega$ as an unbranched $d-sheeted$ covering : $ \omega:\Sigma_g^*\rightarrow H_*^+/PSL_2(\mathbb{Z})$, and can be lifted to a holomorphic map $\tilde{\omega}:\mathbb{H}^+\rightarrow \mathbb{H}^+_*$:
\begin{equation*}
    \xymatrix{
    \mathbb{H}^+\ar@{-->}^{\tilde{\omega}}[rr]\ar_{\pi}[dd] &\   & H_*^+\ar^{P}[dd]\\
    \ &\ &\ \\
    \Sigma_g^*\ar^{\omega}[rr] &\ & H_*^+/PSL_2(\mathbb{Z})
    }
\end{equation*}
where $P$ is the nature projection.

The Jacobian $J$ mentioned above is called the function invariant of $X$, another remarkable invariant is called the homological invariant belongs to $J$, or elliptic monodromy, which can be understood as a
$\pi_1(\Sigma_g^*)$-representation $\rho:\pi_1(o,\Sigma_g^*)\rightarrow SL_2(\mathbb{Z})$ defined as following:

(1) Step 1:choose small loops $\gamma_p(t),0\le t \le 1$ based on $o\in \Sigma_g^*$ such that $\{[\gamma_p], p\in \mathcal{P}\}$ form a basis of $\pi_1(o,\Sigma_g^*)$, let $e_p^1(t),e_p^2(t)$ be continuously varing bases of $H_1(X_{\gamma_p(t)},,\mathbb{Z})$, then:
\begin{equation*}
    (e_p^1(1),e_p^2(1))=(e_p^1(0),e_p^2(0))\cdot
    A_p,A_p=\begin{pmatrix}
        a_p & b_p \\
        c_p & d_p
    \end{pmatrix}\in SL_2(\mathbb{Z})
\end{equation*}

(2) Step 2: The representation $\rho$ is generated by
$[\gamma_p]\rightarrow A_p,p\in \mathcal{P}$.

Note that the representation $\rho$ defines a locally constant sheaf $G'$ over $\Sigma_g^*$ whose stalks are isomorphic to $\mathbb{Z}\oplus \mathbb{Z}$, and $G'$
can be extended to the sheaf $G=\bigcup\limits_{p\in\mathcal{P}} G_{p}\cup G'$ over $\Sigma_g$, where
$G_p:=\Gamma(G'|B^*_{\delta}(p))$, We say the sheaf $G$ belongs to the meromorphic function $J$. Given a meremorphic function $J:\Sigma_g\rightarrow \mathbb{P}^1$, and a sheaf $G$ over $\Sigma_g$ belongs to $J$, there is a unique compact elliptic surface $\Phi:X\rightarrow \Sigma_g$ with globally defined holomorphic section $s$, which posses the functional invariant $J$ and sheaf $G$
over$\Sigma_g$ belongs to $J$, note that the existence of the section $s$ implies
that the elliptic surface free from multiple fibers.

The element $[\gamma_p]\in \Sigma_g^*$ can also be treated as a deck transform of the covering
$\pi$, we have the following commutative diagram:
\begin{equation*}
    \xymatrix{
    \mathbb{H}^+\ar^{\tilde{\omega}}[r]\ar_{\gamma_p}[d] & H_*^+\ar^{\rho([\gamma_p])}[d]\\
    \mathbb{H}^+\ar^{\tilde{\omega}}[r] & H_*^+
    }
\end{equation*}
which means for $[\gamma_p]\in \pi_1(\Sigma_g^*),\rho([\gamma_p])=
\begin{pmatrix}
    a & b\\
    c & d
\end{pmatrix}$,
$\tilde{\omega}(\gamma_p\cdot \zeta)=\frac{a\zeta+b}{c\zeta+d}$.
Let $\mathcal{G}_X$ be the group containing pairs $g_X([\gamma],\textbf{n}),[\gamma]\in \pi_1(\Sigma_g^*),\textbf{n}=(n_1,n_2)\in \mathbb{Z}^2$.
\begin{equation*}
\begin{split}
       g_X([\gamma_1],\textbf{n})\cdot g_X([\gamma_2],\textbf{m}):&=
       g_X([\gamma_1]\cdot [\gamma_2],\textbf{n}\cdot \rho([\gamma_2])+\textbf{m}) \\
       g_X([\gamma],\textbf{n})^{-1}:&=g_X([\gamma]^{-1},-\textbf{n}\cdot \rho([\gamma])^{-1})
\end{split}
\end{equation*}
The product space $\mathbb{H}^+\times \mathbb{C}$ possese a $\mathcal{G}_X$-action defined as:
\begin{equation*}
    \begin{split}
        \mathcal{G}_X\times \mathbb{H}^+\times\mathbb{C} &\rightarrow \mathbb{H}^+\times\mathbb{C}\\
        g_X([\gamma],\textbf{n})\times (\zeta,z) &\rightarrow \left(
        \gamma\cdot \zeta,\frac{z+n_1\tilde{\omega}(\zeta)+n_2}{c\tilde{\omega}(\zeta)+d}
        \right)
    \end{split}
\end{equation*}
Let $\mathcal{N_X}\triangleleft \mathcal{G}_X$ be the normal subgroup of $\mathcal{G}_X$ consists elements of
$\{g(\textbf{1},\textbf{n})\}$, then $\mathbb{H}^+\times_{\tilde{\omega}}\mathbb{C}\cong
(\mathbb{H}^+\times \mathbb{C})/\mathcal{N}_X$ and $\mathcal{G}_X/\mathcal{N}_X\cong
\pi_1(o,\Sigma_g^*)$ we will call the $\pi_1(o,\Sigma_g^*)$-action on $\mathbb{H}^+\times_{\tilde{\omega}} \mathbb{C}$ be the induced $\mathcal{G}_X/\mathcal{N}_X
$-action on $\mathbb{H}^+\times_{\tilde{\omega}}\mathbb{C}$, then $X^*\cong(\mathbb{H}^+\times \mathbb{C})/\mathcal{G}_X\cong (\mathbb{H}^+\times_{\tilde{\omega}}\mathbb{C})/\pi_1(o,\Sigma_g^*)$, . We end this subsection with following  commutative diagram:
\begin{equation}\label{R.D}
    \begin{split}
        \xymatrix{
        \mathbb{H}^+\times_{\tilde{\omega}}\mathbb{C}\ar^{h_{\tilde{\omega}}}[r]\ar_{[\gamma]}[d] & \mathbb{H}^+\times_{Id}\mathbb{C}\ar^{\rho([\gamma])}[d]\\
        \mathbb{H}^+\times_{\tilde{\omega}}\mathbb{C}\ar^{h_{\tilde{\omega}}}[r] & \mathbb{H}^+\times_{Id}\mathbb{C}
        }
    \end{split}
\end{equation}

\subsection{Local structure at singular fibers}
Let $p\in \mathcal{P}$, let $d_p:=mult_pJ$, and let $\mu_p$ be the order of group $G_p=\{A\in PSL_2(\mathbb{Z})|A\cdot\omega(p)=\omega(p)\}$, where $\omega$ is the multi-valued period function on $\Sigma_g^*$.
\begin{equation*}
\begin{split}
 \mu_p=\left\{
\begin{aligned}
1, & \ J(p)\ne 0,1,\infty\\
3, & \ J(p)=0\\
2, & \ J(p)=1\\
\infty, &\ J(p)=\infty
\end{aligned}
\right.
 \end{split}
\end{equation*}
\subsubsection{$J(p)\ne \infty$}
 $h_p:=ord\ \rho([\gamma_p])$ is finite,
we can choice a suitable coordinate $(B_{\delta}(0),\sigma),\sigma(p)=0$ of
$p$ such that
\begin{equation}\label{N-L-M}
    \begin{split}
        \Phi^{-1}(B_{\delta}^*(0))&\cong (B_{\delta^{1/h_p}}^*(0)\times_{\omega_p}\mathbb{C})/\mathbb{Z}_{h_p}\\
        \sigma&=\Phi([\tau,z])=\tau^{h_p}
    \end{split}
\end{equation}
let $e_n:=e^{2\pi \sqrt{-1}/n},\eta=e^{2\pi \sqrt{-1}/3}$ let $g_p$ be the generator $Z_{h_p}$
We have following table in case $J(p)\ne \infty$
\begin{equation*}
 \begin{tabular}{c|c|c|c|c|c|c}
 $X_p$ & $J(p)$ & $h_p$ & $d_p$ & $\rho([\gamma_p])$ & $\omega_p(\tau)$ & $g_p([\tau,z])$\\ \hline

 Regular & $\ne \infty$ & 1 & $\equiv 0\mod \mu_p$ & $\textbf{1}$ & $\tau^{d_p/\mu_p}+\omega_0$  & $[\tau,z]$\\

 $\mathrm{I_0}^*$   & $\ne \infty$ & $2$ & $\equiv 0\mod \mu_p$ & $-\textbf{1}$ & $\tau^{2d_p/\mu_p}+\omega_0$ & $[-\tau,-z]$ \\

 $\mathrm{II}$ & $0$ & $6$ & $\equiv 1\mod 3$ &$\setlength{\arraycolsep}{1.5pt} \begin{pmatrix}
     1 & 1\\
     -1 & 0
 \end{pmatrix}$ & $\frac{\eta-\eta^2\tau^{2d_p}}{1-\tau^{2d_p}}$ &$\left[
 e_6\tau,-\frac{z}{\omega_p(\tau)}
 \right]$\\

 $\mathrm{II^*}$ & $0$ & $6$ & $\equiv 2\mod 3$ &$\setlength{\arraycolsep}{1.5pt} \begin{pmatrix}
     0 & -1 \\
     1 & 1
 \end{pmatrix}$ & $\frac{\eta-\eta^2\tau^{2d_p}}{1-\tau^{2d_p}}$ & $\left[
 e_6\tau,\frac{z}{\omega_p(\tau)+1} \right]$\\

$\mathrm{IV^*}$ & $0$ & $3$ &$\equiv 1\mod 3$ & $\setlength{\arraycolsep}{1.5pt} \begin{pmatrix}
    -1 & -1\\
    1 & 0
\end{pmatrix}$ & $\frac{\eta-\eta^2\tau^{d_p}}{1-\tau^{d_p}}$ &$\left[
 e_3\tau,\frac{z}{\omega_p(\tau)} \right]$\\

 $\mathrm{IV}$ & $0$ & $3$ & $\equiv 2\mod 3$ & $\setlength{\arraycolsep}{1.5pt} \begin{pmatrix}
     0 & 1\\
     -1 & -1
 \end{pmatrix}$ & $\frac{\eta-\eta^2\tau^{d_p}}{1-\tau^{d_p}}$ &$\left[
 e_3\tau,-\frac{z}{\omega_p(\tau)+1} \right]$\\

$\mathrm{III}$ & $1$ & $4$ & $\equiv 1\mod  2$ &$\setlength{\arraycolsep}{1.5pt} \begin{pmatrix}
    0 & 1 \\
    -1 & 0
\end{pmatrix}$ & $\frac{i+i\tau^{2d_p}}{1-\tau^{2d_p}}$ &  $\left[e_4\tau,-\frac{z}{\omega(\tau)}\right]$ \\

$\mathrm{III^*}$ & $1$ & $4$ & $\equiv 1\mod 2$ & $\setlength{\arraycolsep}{1.5pt} \begin{pmatrix}
    0 & -1\\
    1 & 0
\end{pmatrix}$ & $\frac{i+i\tau^{2d_p}}{1-\tau^{2d_p}}$ &
$\left[e_4\tau,\frac{z}{\omega(\tau)}\right]$ \\
\end{tabular}
\end{equation*}
where
\begin{equation*}
  \mu_p=\left\{
\begin{aligned}
1, & \ J(p)\ne 0,1,\infty\\
3, & \ J(p)=0\\
2, & \ J(p)=1
\end{aligned}
\right.
\end{equation*}

\subsubsection{$J(p) =\infty$}
Assume $p$ is a pole of $J$ of order $d_p$:

\textbf{Case 1} $\rho([\gamma_p])=\begin{pmatrix}
    1 & b\\
    0 & 1
\end{pmatrix}$, $X_p$ is of type $\mathrm{I}_b$:
We can choice a suitable coordinate such that
\begin{equation*}
    \begin{split}
        \Phi^{-1}(B_{\delta}^*(p))=(\mathbb{H}^+_{\delta}\times \mathbb{C})/G_p
    \end{split}
\end{equation*}
where $\mathbb{H}^+_{\delta}:=\{Im\zeta>-\frac{1}{2\pi}\log \delta\}\subset \mathbb{H}^+$, $\omega_(\zeta)=d_p\zeta$ and $G_p$ consists pairs:
\begin{equation*}
    \begin{split}
        g(k,n_1,n_2): (\zeta,z)\rightarrow (\zeta+k,z+n_1d_p\zeta+n_2),
        k,n_1,n_2\in\mathbb{Z}
    \end{split}
\end{equation*}

\textbf{Case 2} $\rho([\gamma_p])=\begin{pmatrix}
    -1 & -b\\
    0 & -1
\end{pmatrix}$, $X_p$ is of type $\mathrm{I}_b^*$:
We can choice a suitable coordinate such that
\begin{equation*}
    \begin{split}
        \Phi^{-1}(B_{\delta}^*(p))=(\mathbb{H}^+_{\delta}\times \mathbb{C})/G_p
    \end{split}
\end{equation*}
$\omega_(\zeta)=d_p\zeta$ and $G_p$ consists pairs:
\begin{equation*}
    \begin{split}
        g(k,n_1,n_2): (\zeta,z)\rightarrow (\zeta+k,(-1)^k(z+n_1d_p\zeta+n_2)),
        k,n_1,n_2\in\mathbb{Z}
    \end{split}
\end{equation*}

\subsection{$H^{1,1}(X)\cap H^2(X,\mathbb{R})$}
This subsection is a essential review of the lattice structure of $H^{1,1}(X)\cap H^2(X,\mathbb{R})$ and one can find more details in \cite{Shioda90}.
$H^{1,1}_{\mathbb{R}}(X):=H^{1,1}(X)\cap H^2(X,\mathbb{R})$
can be decomposed with parts:
\begin{equation}\label{decomposition of H_2}
    H^{1,1}_{\mathbb{R}}=(NS(X)\otimes \mathbb{R})\oplus(NS(X)\otimes \mathbb{Q})^{\perp}
\end{equation}
where $NS(X)$ is the N\'eron-Severi group, by the Lefschetz theorem on $(1,1)$-classes,
\begin{equation*}
    H^{1,1}(X)\cap H^2(X,\mathbb{Z})\cong NS(X)
\end{equation*}
Via Poincaré duality, the sub-lattice  $NS(X)\otimes \mathbb{R}$ corresponds to the real vector space spanned by the homology classes of algebraic divisors, and the sub-lattice
$(NS(X)\otimes \mathbb{Q})^{\perp}:=\{\alpha\in   H^{1,1}_{\mathbb{R}}| \alpha\cdot \beta=0, \forall \beta\in NS(X)\otimes \mathbb{Q}\}$

\begin{equation*}
    NS(X)=T(X)\oplus T(X)^{\perp}
\end{equation*}
where $T(X)$ called the trivial sublattice of $NS(X)$, is generated by:

(1) zero section $[s]$.

(2) fiber class $[F]$.

(3) The classes $\{[\Theta_{pi}] \}_{i=1}^{m_p-1}$  which arise from the reducible singular fiber $X_p=\Theta_{p0}+\sum\limits_{i=1}^{m_p-1}n_{pi}\Theta_{pi}$ ($\Theta_{i0}$ is the component meet zero section $s$).

$T^{\perp}$ is called the essential sublattice of $NS(X)$ consists of all elements of $NS(X)$ that have zero intersection with every element of $T$.
For each section $s_i\in MW(X/\Sigma_g)$, there is a unique element $\psi(s_i)\in T^{\perp}$ such that $[s_i]\equiv \psi(s_i)\mod T$, and
$\psi:MW(X)\rightarrow T^{\perp}$ called Shioda homomorphism induces
an isomorphism:
\begin{equation*}
\begin{split}
     \psi:MW(X)\otimes \mathbb{Q}&\xrightarrow{\cong} T^{\perp}\otimes \mathbb{Q}\\
     s_i&\rightarrow [s_i]-[s]-([s_i]\cdot [s]+\chi)[F]
\end{split}
\end{equation*}

let $MW^0(X/\Sigma_g)=\{\sigma\in MW(X/\Sigma_g)| \sigma\ meets\ \Theta_{p0},\forall p\}$, then we
can choice $s_1,...,s_r\in MW^0(X/\Sigma_g)$ as a $\mathbb{Q}$-basis of $MW(X/\Sigma_g)\otimes \mathbb{Q}$, the following is a
basis of $NS(X)\otimes \mathbb{Q}$:
\begin{equation*}
    [s],[F],\{[\Theta_{pi}]\}_{i=1}^{m_p-1},\{\psi(s_i)\}_{i=1}^r
\end{equation*}
and the intersection matrix with this basis is
\begin{equation*}
    M_a=
\begin{pmatrix}
-\chi & 1      &        &        \\
1     & 0      &        &         \\
      &        & \bigoplus_p R_p  \\
      &        &          & E
\end{pmatrix}
\end{equation*}
where $\chi=h^0(X,\mathcal{O}_X)-h^1(X,\mathcal{O}_X)+h^2(X,\mathcal{O}_X)$ is the Euler Poincar\'e character of $X$, and $R_p=(\Theta_{pi}\cdot \Theta_{pj})_{(m_{p}-1)\times (m_{p}-1)}$, $E=(\psi(s_i)\cdot \psi(s_j))_{r\times r}$. Indeed, we have following explicit computation for $E$(\cite{Shioda90}):
\begin{theorem}[Shioda, 1990]
    \begin{equation*}
    \begin{split}
        -\psi(s_i)\cdot \psi(s_j)&=\chi+s_i\cdot s+s_j\cdot s-s_i\cdot s_j\\
        -\psi(s_i)\cdot \psi(s_j)&=2\chi+2s_i\cdot s
    \end{split}
    \end{equation*}
\end{theorem}


\section{Construction of semi-flat cscK metric on $X^*$}
The main goal of this section is to construct a semi-flat K\"ahler metric  with constant scalar curvature (semi-flat cscK) on $\mathbb{H}^+\times_{Id}\mathbb{C}$ and $X^*$.

We introduce the coordinates $(\zeta,z)=(u+\sqrt{-1}v,x+\sqrt{-1}y)\in \mathbb{H}^+\times\mathbb{C}$, the desired metric on $\mathbb{H}^+\times_{Id}\mathbb{C}$, denoted by $\eta$, can be expressed in these coordinates as:
\begin{equation}\label{eta}
    \begin{split}
 \eta&=\frac{\sqrt{-1}}{2}
 (
 A d\xi\wedge d\bar{\xi}
 +B d\xi\wedge d\bar{z}+
 \bar{B}dz\wedge d\bar{\xi}
 +Cdz\wedge d\bar{z}
 )
 \end{split}
\end{equation}
where $A,B,C$ are functions of the coordinates, to ensure fiber volume $1$, we require $C=1/v$, the K\"ahler condition reads:
\begin{equation}
    \begin{split}
         \frac{\partial A}{\partial z}&=
 \frac{\partial\bar{B}}{\partial\zeta}\\
 \frac{\partial C}{\partial\zeta}&=
 \frac{\partial B}{\partial z}
    \end{split}
\end{equation}
The metric is required to be $\mathbb{Z}^2$-invariant:
\begin{equation*}
        g(\textbf{1},n,m)^*\eta=\eta,\forall n,m\in \mathbb{Z}\Leftrightarrow
    \end{equation*}
\begin{equation}\label{a}
         A(\zeta,z+1)=A(\zeta,z)
\end{equation}
\begin{equation}\label{b}
        B(\zeta,z+1)=B(\zeta,z)
\end{equation}
\begin{equation}\label{c}
        A(\zeta,z+\zeta)+B(\zeta,z+\zeta)+
 \bar{B}(\zeta,z+\zeta)+C(\zeta)=A(\zeta,z)
\end{equation}
\begin{equation}\label{d}
        B(\zeta,z+\zeta)+C(\zeta)=B(\zeta,z)
\end{equation}

Since $\frac{\partial^2\bar{B}}{\partial z\partial \bar{z}}=
\frac{\partial}{\partial z}
\frac{\partial C}{\partial\bar{\zeta}}=0$ so
$\frac{\partial \bar{B}}{\partial z}$ is holomorphic on $\Phi^{-1}(\zeta)$, combining with \ref{b} and \ref{d}, we see that $\frac{\partial \bar{B}}{\partial z}(\zeta,z)$ is double-period in $z$,
thus, $B(\zeta,z)$ reduce a function of $\zeta$, denoted by $Q(\zeta)$. Therefore:
\begin{equation}\label{e}
 \begin{split}
 B(\zeta,z)&=\beta(\zeta)+\frac{\partial C}{\partial \zeta}z+\bar{Q}\bar{z}\\
 \beta(\zeta):&=B(\zeta,0)
 \end{split}
\end{equation}
again by \ref{b} and \ref{e},we have:
\begin{equation}
    \frac{\partial C}{\partial\zeta}+\bar{Q}=0
\end{equation}
Therefore
\begin{equation}
 \begin{split}
 B(\zeta,z)&=\beta(\zeta)+
 \frac{\partial C}{\partial\zeta}z-\frac{\partial C}{\partial\zeta}\bar{z}\\
 &=\beta(\zeta)-\frac{y}{v^2}
 \end{split}
\end{equation}
Now,we compute $A$:
\begin{equation}
 \begin{split}
 \frac{\partial A}{\partial z}&=
 \frac{\partial\bar{B}}{\partial\zeta}
 =\frac{\partial\bar{\beta}}{\partial\zeta}
 -\frac{\sqrt{-1} y}{v^3}\\
 \frac{\partial A}{\partial \bar{z}}&=
 \frac{\partial B}{\partial\zeta}
 =\frac{\partial\beta}{\partial\bar{\zeta}}
 +\frac{\sqrt{-1}y}{v^3}\\
 \frac{\partial^2 A}{\partial z^2}&=
 -\frac{1}{2v^3}
=\frac{\partial^2A}{\partial \bar{z}^2}\\
\frac{\partial^2A}{\partial z\partial\bar{z}}
&=\frac{1}{2v^3}
 \end{split}
\end{equation}
Therefore,we have
\begin{equation}
 \begin{split}
 A(\zeta,z)=\alpha(\zeta)+
 \frac{\partial\bar{\beta}}{\partial\zeta} z
 +\frac{\partial\beta}{\partial\bar{\zeta}}\bar{z}
 -\frac{1}{4v^3}(z-\bar{z})^2\\
 =\alpha(\zeta)+
 \frac{\partial\bar{\beta}}{\partial\zeta} z
 +\frac{\partial\beta}{\partial\bar{\zeta}}\bar{z}
 +\frac{ y^2}{v^3}
 \end{split}
\end{equation}
By \ref{a}, we have
\begin{equation}
 \begin{split}
 \frac{\partial\bar{\beta}}{\partial\zeta}
 +\frac{\partial\beta}{\partial\bar{\zeta}}
 =0
 \end{split}
\end{equation}
By \ref{c}, we have:
\begin{equation}
 \begin{split}
 \frac{\partial \bar{\beta}}{\partial\zeta}
\zeta+
\frac{\partial\beta}{\partial\bar{\zeta}}\bar{\zeta}+
\beta+\bar{\beta}=0
 \end{split}
\end{equation}
As a result:
\begin{equation}
 \frac{\partial\bar{\beta}}{\partial\zeta}=\frac{\sqrt{-1}(\beta+\bar{\beta})}{2v}
\end{equation}
As a summary:
\begin{equation}\label{ABCD}
 \begin{split}
 C(\zeta,z)&=\frac{1}{v}\\
 B(\zeta,z)&=\beta(\zeta)-\frac{y}{v^2}\\
 A(\zeta,z)&
  =\alpha(\zeta)+
 \frac{\partial\bar{\beta}}{\partial\zeta}z
 +\frac{\partial\beta}{\partial\bar{\zeta}}\bar{z}
 +\frac{y^2}{v^3}\\
 &=\alpha(\zeta)-
 \frac{y}{v}(\beta(\zeta)+\bar{\beta}(\zeta))
 +\frac{y^2}{v^3}\\
 \end{split}
\end{equation}
the volume form of $\eta$ is:
\begin{equation*}
    vol=\left(\frac{\alpha(\zeta)}{v}-
|\beta(\zeta)|^2 \right)du\wedge dv\wedge dx\wedge dy
\end{equation*}
so
\begin{equation}\label{vol}
    \alpha (\zeta)>v|\beta(\zeta)|^2
\end{equation}
the scalar curvature is:
\begin{equation}\label{csc}
    s=-\frac{1}{\alpha-v|\beta|^2}
\triangle_{\zeta}\log\left(\frac{\alpha}{v}
-|\beta|^2\right)
\end{equation}
where $\triangle_{\zeta}=\frac{\partial^2}{\partial u^2}
+\frac{\partial^2}{\partial v^2}$,
let $s$ be the section $\mathbb{H}^+\rightarrow \mathbb{H}^+\times_{Id}\mathbb{C}$ determined by $s(\zeta)=[\zeta,0]$, then the metric $s^*\eta$ on $\mathbb{H}^+$ can be written as:
\begin{equation*}
    s^*\eta(\zeta)=\frac{\sqrt{-1}}{2}\alpha(\zeta)d\zeta
    \wedge d\bar{\zeta}
\end{equation*}
when $\alpha(\zeta)=1/v^2$, it's the classical Poincaré half upper plane.

When $\beta=0,\alpha=1/v^2$, we obtain the following K\"ahler metric on $\mathbb{H}^+\times\mathbb {C}$ :
\begin{equation}\label{SF-cscK}
    \begin{split}
        \eta=\sqrt{-1}\partial\bar{\partial}
        \left(
                     \log\frac{1}{v^2}+\frac{y^2}{v}
        \right)
    \end{split}
\end{equation}

\begin{lemma}\label{G.I}
    the K\"ahler metric $\eta$ on $\mathbb{H}^+\times \mathbb{C}$ defined as \ref{SF-cscK} is $\mathcal{G}$-invariant, therefore is a well-defined $SL_2(\mathbb{Z})$-invariant K\"ahler metric on
    $\mathbb{H}^+\times_{Id}\mathbb{C}$, moreover it is semi-flat with constant scalar curvature.
\end{lemma}
\begin{proof}
    We rewrite $\eta$ as:
\begin{equation}
 \begin{split}
\eta=\underbrace{\frac{\sqrt{-1}}{2}\frac{d\zeta\wedge d\bar{\zeta}}{v^2}}_{\eta_1}
+\underbrace{\frac{\sqrt{-1}}{2}
[d\zeta,dz]\wedge\begin{pmatrix}
\frac{y^2}{v^3}&-\frac{y}{v^2}\\-\frac{y}{v^2}&
\frac{1}{v}
\end{pmatrix}
\begin{pmatrix}
d\bar{\zeta}\\d\bar{z}
\end{pmatrix}}_{\eta_2}
 \end{split}
\end{equation}
the properties of semi-flat and constant scalar curvature can be checked by direct computation:
\begin{equation*}
    \begin{split}
       &1,volume\ of\ fibers\equiv1 \\
       &2,scalar\ curvature\equiv -3
    \end{split}
\end{equation*}
we prove the part of invariant-property: we notice that $\eta_1=\frac{\sqrt{-1}}{2}\frac{d\zeta\wedge d\bar{\zeta}}{v^2}$ is just the pullback of the hyperbolic metric on the base space $\mathbb{H}^+$, since the hyperbolic metric is $SL_2(\mathbb{Z})-invariant$ so the first part $\eta_1$ is $\mathcal{G}-invariant$, so it suffices to check the $\mathcal{G}-invariant$ of $\eta_2$. Denote
\begin{equation*}
    [\eta_2]=M(\zeta,z):=\begin{pmatrix}
\frac{y^2}{v^3}&-\frac{y}{v^2}\\-\frac{y}{v^2}&
\frac{1}{v}
\end{pmatrix}
\end{equation*}
for each $g\in \mathcal{G}$. Let $J_g$ be the Jacobian matrix of
$g$, i.e.
\begin{equation*}
    \begin{split}
        g(\zeta,z)&=(g_1(\zeta,z),g_2(\zeta,z))\\
        J_g:&=\begin{pmatrix}
\frac{\partial g_1}{\partial\zeta}
&\frac{\partial g_2}{\partial\zeta}
\\ \frac{\partial g_1}{\partial z}
&\frac{\partial g_2}{\partial z}
\end{pmatrix}
    \end{split}
\end{equation*}
The $\mathcal{G}$-invariant of $\eta_2$ is equivalent to
\begin{equation*}
  J_g(M\circ g)\bar{J}_g^t=M,\forall g\in \mathcal{G}
\end{equation*}
Since $g(A,(m,n))=g(A,(0,0))\cdot g(1,(1,0))^m\cdot g(1,(0,1))^n,
\forall A\in SL_2(\mathbb{Z}),m,n\in Z^2$, and $SL_2(\mathbb{Z})$ is generated by
$$T=\begin{pmatrix}
1&1\\0&1
\end{pmatrix},S=\begin{pmatrix}
0&1\\-1&0
\end{pmatrix}
$$
so, it suffices to check the invariant property for elements:
$$
g(T,\textbf{0}),g(S,\textbf{0}),
g(1,(1,0)),g(1,(0,1))
$$
The following is a direct computation:
\begin{equation*}
 \begin{split}
 [g(T,\textbf{0})^*\eta_2]&=M
 \end{split}
\end{equation*}

\begin{equation*}
 \begin{split}
  [g(S,\textbf{0})]^*\eta_2]&=
  \begin{pmatrix}
\frac{1}{\zeta^2}&
\frac{z}{\zeta^2}
\\0&-\frac{1}{\zeta}
\end{pmatrix}
\begin{pmatrix}
\frac{\left(
\frac{xv-yu}{u^2+v^2}
\right)^2}
{\left(
\frac{v}{u^2+v^2}
\right)^3}
&
-\frac{\frac{xv-yu}{u^2+v^2}}
{\left(
\frac{v}{u^2+v^2}
\right)^2}
\\
-\frac{\frac{xv-yu}{u^2+v^2}}
{\left(
\frac{v}{u^2+v^2}
\right)^2}
&
\frac{1}{\frac{v}{u^2+v^2}}
\end{pmatrix}
\begin{pmatrix}
\frac{1}{\bar{\zeta}^2}
&0\\ \frac{\bar{z}}{\bar{\zeta}^2}
&-\frac{1}{\bar{\zeta}}
\end{pmatrix}\\
&=\frac{1}{u^2+v^2}\begin{pmatrix}
1&z\\0&-\zeta
\end{pmatrix}
 \begin{pmatrix}
\frac{(xv-yu)^2}{v^3}
&-\frac{xv-yu}{v^2}\\
-\frac{xv-yu}{v^2}&\frac{1}{v}
\end{pmatrix}
\begin{pmatrix}
1&0\\ \bar{z}&-\bar{\zeta}
\end{pmatrix}\\
&=M
 \end{split}
\end{equation*}

\begin{equation*}
 \begin{split}
 [g(1,(0,1))^*\eta_2]=M
 \end{split}
\end{equation*}

\begin{equation*}
 \begin{split}
 [g(1,(1,0))^*\eta_2]&=
  \begin{pmatrix}
1&1\\0&1
\end{pmatrix}
\begin{pmatrix}
\frac{(y+v)^2}{v^3}&-\frac{y+v}{v^2}\\-\frac{y+v}{v^2}&
\frac{1}{v}
\end{pmatrix}
\begin{pmatrix}
1&0\\1&1
\end{pmatrix}\\
&=M
 \end{split}
\end{equation*}
These complete the proof.
\end{proof}

\begin{theorem}\label{universal-rigidity}
   There is a unique $SL_2(\mathbb{Z})$-invariant semi-flat cscK metric on $\mathbb{H}^+\times_{Id}\mathbb{C}$ with fiber volume $1$ and scalar curvature $-3$.
\end{theorem}
\begin{proof}
    it suffices to show the uniqueness, let the desired metric
$\eta$ given as \ref{eta} with coefficients satisfies \ref{ABCD}, we need to show
$\beta(\zeta)=0$ and $\alpha(\zeta)=1/v^2$.
Consider the invariant property of $\eta$ under the transform:
\begin{equation}
    [\zeta,z]\rightarrow [-1/\zeta,-z/\zeta]
\end{equation}
we compare coefficients of $d\zeta\wedge d\bar{\zeta}$:
\begin{equation}\label{zeta}
\begin{split}
   & \alpha(-1/\zeta)-\frac{2(yu-xv)}{v}\beta_1(-1/\zeta)
+2(\beta_1(-/\zeta)x+\beta_2(-1/\zeta)y)\\
&=\alpha(\zeta)|\zeta|^4\\
\end{split}
\end{equation}
take $z=x+\sqrt{-1}y=0$ on \ref{zeta} we conclude:
\begin{equation}\label {alpha2}
    \alpha(-1/\zeta)=\alpha(\zeta)|\zeta|^4
\end{equation}
therefore
\begin{equation}
    -\frac{2(yu-xv)}{v}\beta_1(-1/\zeta)
+2(\beta_1(-/\zeta)x+\beta_2(-1/\zeta)y)=0
\end{equation}
take $y=0$, we have:
\begin{equation}
    4x\beta_1(-1/\zeta)=0,\forall \zeta,x \Rightarrow \beta_1(\zeta) \equiv 0
\end{equation}
so
\begin{equation}
    2\beta_2(-1/\zeta)y=0\forall \zeta,y\Rightarrow \beta_2(\zeta)=0
\end{equation}
as a result $\beta(\zeta)=\beta_1+\sqrt{-1}\beta_2 \equiv 0$.

Finally, we show $\alpha(\zeta)=1/v^2$: let $f:=v^2\alpha(\zeta)$, then \ref{alpha2}
implies:
\begin{equation}\label{f1}
    f(\zeta)=f(-1/\zeta)
\end{equation}
moreover, the invariant of $\eta$ under transform: $[\zeta,z]\rightarrow [\zeta+1,z]$ implies:
\begin{equation}\label{f2}
    f(\zeta+1)=f(\zeta)
\end{equation}
the scalar curvature equation \ref{csc} reads:
\begin{equation}
    \begin{split}
        &\triangle \log f=\frac{3(f-1)}{v^2}\\
        &f\in C^{\infty}(\mathbb{H}^+), f>0.
    \end{split}
\end{equation}
If $f(u,v)=f(v)$ is independent of $u=Re\zeta$, then \ref{f1} implies:
\begin{equation}
    \begin{split}
        f(v)&=f(v/(u^2+v^2)),\forall (u,v)\in \mathbb{H}^+\\
        take\ u=0: f(v)&=f(1/v)\\
        take\ u=\kappa v: f(v)&=f\left(\frac{1}{(\kappa^2+1)v}\right)=f((\kappa^2+1)v),\forall \kappa \in \mathbb{R}
    \end{split}
\end{equation}
as a result $f=constant$ and the equation \ref{csc} implies $f\equiv 1$.

Therefore, to show $f\equiv 1$, it suffices to show that $f(u,v)$ is independent
of $u$. Let $D_{uv}(\epsilon)\subset \mathbb{H}^+$ be the rectangle domain:
\begin{equation*}
    D_{uv}(\epsilon):=\{(\hat{u},\hat{v}),u\le
    \hat{u}\le u+1, v\le \hat{v} \le v+\epsilon \}
\end{equation*}

let
\begin{equation*}
    \Omega_{uv}(\epsilon)=D_{uv}(\epsilon)\times_{Id}\mathbb{C}\subset
    \mathbb{H}^+\times_{Id}\mathbb{C}
\end{equation*}
The invariant of $\eta$ under transision $[\zeta,z]\rightarrow [\zeta+1,z]$
implies that the volume $V(u,v,\epsilon):=vol(\Omega_{uv}(\epsilon))$ is independent of $u$, i.e. $\frac{\partial V(u,v,\epsilon)}{\partial u}=0$.
The volume form of $\eta$ is computed as
$vol(\eta)=\frac{\alpha}{v}du\wedge dv\wedge dx\wedge dy$.
\begin{equation*}
    \begin{split}
        V(u,v,\epsilon)&=\int_{\Omega_{uv}(\epsilon)}
        \frac{\alpha(\zeta)}{v}
        du\wedge dv\wedge dx\wedge dy\\
        &=\int_{D_{uv}(\epsilon)}\frac{f(\zeta)}{v^2}du\wedge dv\\
        &=\frac{1}{3}\int_{D_{uv}(\epsilon)}\triangle \log f du\wedge dv
        +\int_{D_{uv}(\epsilon)}\frac{1}{v^2}du\wedge dv\\
        &=\frac{1}{3}\left(\int_{\partial D_{uv}(\epsilon)}\triangledown \log f\cdot ds\right)
        +\frac{1}{v}-\frac{1}{v+\epsilon}\\
        &=\frac{1}{3}\left(-\int_{u}^{u+1}\frac{\partial\log f}{\partial v}
        (\hat{u},v)d\hat{u}+\int_{v}^{v+\epsilon}\frac{\partial\log f}{\partial u}
        (u+1,\hat{v})d\hat{v}\right.\\
       &\ \ \ +  \left.\int_{u+1}^{u}\frac{\partial\log f}{\partial v}(\hat{u},v+\epsilon)
        d\hat{u}-\int_{v+\epsilon}^{v}\frac{\partial\log f}{\partial u}
        (u,\hat{v})d\hat{v}\right)+\frac{1}{v}-\frac{1}{v+\epsilon}\\
       &=\frac{1}{3}\left(-\int_{u}^{u+1}\frac{\partial\log f}{\partial v}
        (\hat{u},v)d\hat{u}+2\int_{v}^{v+\epsilon}\frac{\partial\log f}{\partial u}
        (u,\hat{v})d\hat{v}\right.\\
        &\ \ \ +\left.\int_{u+1}^{u}\frac{\partial\log f}{\partial v}(\hat{u},v+\epsilon)
        d\hat{u}\right)+\frac{1}{v}-\frac{1}{v+\epsilon}
    \end{split}
\end{equation*}
as a result we have:
\begin{equation}\label{vol-in}
    \begin{split}
        0&=\frac{\partial V(u,v,\epsilon)}{\partial u}
        =\frac{2}{3}\int_{v}^{v+\epsilon}\frac{\partial^2\log f}{\partial u^2}
        (u,\hat{v})d\hat{v}
    \end{split}
\end{equation}
since \ref{vol-in} is valid for arbitrary $v,\epsilon>0$, so
\begin{equation}
    \frac{\partial^2\log f}{\partial u^2}(u,v)\equiv 0
\end{equation}
Therefore $\log f$ is linear w.r.t $u$, since it is also period w.r.t $u$, so $f(u,v)=f(v)$, as desired.
\end{proof}

Now we try to construct a semi-flat cscK metric on $X^*$: consider the following metric defined on $\mathbb{H}^+\times_{\tilde{\omega}}\mathbb{C}$:
\begin{equation}\label{etaX}
    \eta_X:=h_{\tilde{\omega}}^*\eta
\end{equation}
Combining commutative diagram \ref{R.D} and Lemma \ref{G.I}, we obtain the following result:
\begin{corollary}\label{pi1invariant}
    $\eta_X$ is a $\pi_1(o,\Sigma_g^*)$-invariant semi-flat cscK metric on $\mathbb{H}^+\times_{\tilde{\omega}}\mathbb{C}$, therefore inherits to a well-defined semi-flat cscK metric on $X^*$
\end{corollary}
\begin{proof}
     The properties of semi-flat and constant scalar curvature is a direct result of "pulling-back" construction. it suffices to check the $\pi_1(o,\Sigma_g^*)$-invariant of $\eta_X$: $\forall [\gamma]\in \pi_1(o,\Sigma_g^*)$, we have
\begin{equation*}
    \begin{split}
        [\gamma]^*\eta_X&=[\gamma]^*h_{\tilde{\omega}}^*\eta=(h_{\tilde{\omega}}\circ [\gamma])^*\eta\\
        &=([\gamma]\circ h_{\tilde{\omega}})^*\eta(use\ the\ diagram \ref{R.D})\\
        &=h_{\tilde{\omega}}^*[\gamma]^*\eta
        =h_{\tilde{\omega}}^*\eta=\eta_X(ues\ the\ lemma\ref{G.I})
    \end{split}
\end{equation*}
\end{proof}
\textbf{Remark:} Consider the metric $\eta(\delta,\epsilon)$
on $\mathbb{H}^+\times_{Id}\mathbb{C}$ defined by
\begin{equation}\label{eta delta epsilon}
\eta(\epsilon):=\sqrt{-1}\partial\bar{\partial}\left(-\delta
    \log v^2+\frac{\epsilon y^2}{v}
    \right)
\end{equation}
replacing $\eta$ with $\eta(\delta,\epsilon)$ and defining
$\eta_{X}(\delta,\epsilon):=h_{\tilde{\omega}}^*\eta(\delta,\epsilon)$, we obtian a family of semi-flat cscK metrics $\{X^*,\eta_{X}(\delta,\epsilon)\}$ parametrized by $\delta,\epsilon$,
each has fiber volume $\epsilon/\delta$ and scalar curvature $-3/\delta$.


\section{Uniqueness of semi-flat cscK metric}
This section investigates the uniqueness of the semi-flat cscK metric on a minimal elliptic surface with section.

\subsection{Decomposition of the fibration automorphism group}
Let $\Phi:X\rightarrow \Sigma_g$ be a minimal elliptic surface with globally defined holomorphic section $s:\Sigma_g\rightarrow X$. we set:
\begin{equation*}
    \begin{split}
        H(X,\Phi)&:=\{(h',h)\in Aut(X)\times Aut(\Sigma_g)| \Phi\circ h'=h\circ \Phi\}\\
        H(X,\Phi,s)&:=\{(h',h)\in Aut(X)\times Aut(\Sigma_g)| \Phi\circ h'=h\circ \Phi,h'\circ s=s\circ h\}\\
    \end{split}
\end{equation*}

Denote $MW(X/\Sigma_g,s)$ be the Mordell-Weil group of holomorphic sections $\sigma:\Sigma_g\rightarrow X$ with zero element $s$. Given $\sigma\in MW(X/\Sigma_g,s)$, we can define
$T_{\sigma}\in Aut(X)$ by
\begin{equation*}
    T_{\sigma}: q\rightarrow q+\sigma(\phi(q))
\end{equation*}
Note that $T_s=Id_X$.

\begin{proposition}
    $\forall$ $(h',h)\in H(X,\Phi)$, there is a unique $\sigma\in MW(X/\Sigma_{g},s)$ and $(h'',h)\in H(X,\Phi,s)$, such that
    \begin{equation*}
        (h',h)=(T_{\sigma}\circ h'',h)
    \end{equation*}
\end{proposition}
\begin{proof}
    Let $\sigma=h'\circ s\circ h^{-1}$, set $h''=T_{-\sigma}\circ h'$:
 \begin{equation*}
        \xymatrix{
        X\ar[r]^{h'}\ar[d]^{\Phi}
        & X\ar[r]^{T_{-\sigma}}\ar[d]^{\Phi} & X\ar[ld]_{\Phi}\\
        \Sigma_g\ar@/^/[u]^{s}\ar[r]^{h} & \Sigma_g\ar@/_1pc/[ur]_{s}\ar@/^/[u]^{\sigma}
        & \
        }
 \end{equation*}
   we claim that $(h'',h)\in H(X,\Phi,s)$:
(1) check $\Phi\circ h''=h\circ\Phi$:
\begin{equation*}
    \begin{split}
        \Phi\circ h''=\Phi\circ T_{-\sigma}\circ h'=\Phi\circ h'=h\circ\Phi
    \end{split}
\end{equation*}
(2) check $h''\circ s=s\circ h$, $\forall p\in \Sigma_g$:
    \begin{equation*}
        \begin{split}
            h''\circ s(p)&=T_{-\sigma}\circ h'\circ s(p)\\
            &=h'(s(p))-\sigma (\Phi\circ h'\circ s(p))\\
            &=h'(s(p))-\sigma(h\circ \Phi\circ s(p))\\
            &=h'(s(p))-\sigma (h(p))\\
            &=h'(s(p))-h'\circ s\circ h^{-1}(h(p))\\
            &=h'(s(p))-h'(s(p))\\
            &=s(h(p))
        \end{split}
    \end{equation*}
as a result, we proved the existence of the decomposition:
\begin{equation*}
    (h',h)=(T_{\sigma}\circ h'',h)
\end{equation*}
To show the uniqueness, assume
\begin{equation*}
    (h',h)=
(T_{\sigma'}\circ h''',h),\sigma'\in MW(X/\Sigma_g,s), (h''',h)\in H(X,\Phi,s)
\end{equation*}
is another decomposition, then
$T_{\sigma-\sigma'}=(h'''\circ h''^{-1},h)\in H(X,\Phi,s)$, so
\begin{equation*}
    \sigma-\sigma'=s
\end{equation*}
therefore, $\sigma=\sigma'$ and $h'''=h''$.
\end{proof}
\textbf{Remark:} The decomposition $(h',h)=(h''\circ T_{\sigma},h)$ is generally not possible (or loss of uniqueness), since $h'$ is not necessarily a group homomorphism on the fibers.

\subsection{Invariance and uniqueness}
Recall that $S\subset \mathbb{P}^1$ is a finite set
containing critical values of $J$ and $0,1,\infty$ such that
$\Phi^{-1}(\sigma)$ is regular $\forall \sigma\notin J^{-1}(S)$. Denote $\mathbb{P}^1_*=\mathbb{P}^1\setminus S,\mathcal{P}=J^{-1}(S),\Sigma_g^*=\Sigma_g\setminus \mathcal{P},X^*=\Phi^{-1}(\Sigma_g^*),\mathbb{H}^+_*=\mathbb{H}^+\setminus j^{-1}(S)$, where $\mathbb{H}^+$
is the upper half plane, and $j:\mathbb{H}^+\rightarrow \mathbb{C}$ is the elliptic modular function with
$j(e^{2\pi \sqrt{-1}/3})=0,j(\sqrt{-1})=1$. Fix $\zeta_0\in \mathbb{H}^+,\sigma_0=\pi (\zeta_0), s_0=J(\sigma_0), \xi_0\in j^{-1}(s_0)$, then there is a unique holomorphic map $\tilde{\omega}:\mathbb{H}^+\rightarrow \mathbb{H}^+_*$ with $\tilde{\omega}(\zeta_0)=\xi_0$, such that $j\circ \tilde{\omega}=J\circ \pi$:
\begin{equation*}
    \xymatrix{
    \mathbb{H}^+\ar_{\pi}[d]\ar @{-->}^{\tilde{\omega}}[rr]& \ & \mathbb{H}^+_*\ar^{j}[d]\\
    \Sigma_g^*\ar^{J}[rr]& \ & \mathbb{P}^1_*\\
    }
\end{equation*}

\begin{proposition}\label{HXinvariant}
    The semi-flat cscK metric $\eta_X$ we constructed as (\ref{etaX}) is $H(X,\Phi,s)$-invariant, i.e. $\forall(h',h)\in H(X,\Phi,s)$, we have:
    \begin{equation*}
        \begin{split}
            h'^*\eta_X=\eta_X
        \end{split}
    \end{equation*}
\end{proposition}

\begin{proof}
    Without loss of generality, assume $[\zeta,0]\in \mathbb{H}^+\times_{\tilde{\omega}}\mathbb{C}$ treat an element of $X^*$, corresponding the section
    $s(\pi(\zeta))$:
    \begin{equation*}
        \xymatrix{
        \mathbb{H}^+\times_{Id}\mathbb{C}\ar@{-->}[rr]^{ h_{\pm A}} &\ & \mathbb{H}^+\times_{Id} \mathbb{C}\\
        \mathbb{H}^+\times_{\tilde{\omega}} \mathbb{C}\ar[d]_{P}\ar[u]^{h_{\tilde{\omega}}} \ar@{-->}[rr]^{\bar{h}} &\ & \mathbb{H}^+\times_{\tilde{\omega}} \mathbb{C}\ar[d]^{P}\ar[u]_{h_{\tilde{\omega}}} \\
        X^*\ar[d]_{\Phi}\ar[rr]^{h'} &\ & X^*\ar[d]^{\Phi}\\
        \Sigma_g^*\ar[rr]^{h} &\ & \Sigma_g^*
        }
    \end{equation*}

since $(h'\circ P)_*(\pi_1(\mathbb{H}^+\times_{\tilde{\omega}}\mathbb{C})=P_*(\pi_1(\mathbb{H}^+\times_{\tilde{\omega}}\mathbb{C})$, there is a lift $\bar{h}$ of $h'$ satisfies:
\begin{equation*}
    P\circ \bar{h}=h'\circ P
\end{equation*}
Assume
\begin{equation*}
    h([\zeta,z])=[h_1(\zeta),h_2(\zeta,z)]
\end{equation*}
since $\bar{h}$ restrict to fiberwise isomorphism, we have:
\begin{equation*}
    j\circ\tilde{\omega}\circ h_1(\zeta))=j\circ \tilde{\omega}(\zeta)
\end{equation*}
where $j:\mathbb{H}^+\rightarrow \mathbb{C}$ is the elliptic modular function, therefore, there exists a constant matrix $A=
\begin{pmatrix}
    a & b\\
    c & d
\end{pmatrix}\in SL_2(\mathbb{Z})$, such that
\begin{equation*}
    \tilde{\omega}\circ h_1(z)=A\cdot \tilde{\omega}(\zeta)=\frac{a\tilde{\omega}(\zeta)+b}{c\tilde{\omega}(\zeta)+d}
\end{equation*}
since $h'$ preseve the section $s$, so $h_2(\zeta,0)=0$, and $h_2$ restrict on the fiber
$\zeta\times_{\tilde{\omega}(\zeta)}\mathbb{C}$ an isomorphism which preserve the zero element, as a result:
\begin{equation*}
    h_2(\zeta,z)=\pm\frac{z}{c\tilde{\omega}(\zeta)+d}
\end{equation*}

Let $h_{\pm A}$ be the action of $A$ on $\mathbb{H}^+\times_{Id}\mathbb{C}$ defined as
\begin{equation*}
    h_{\pm A}([\zeta,z])= \left[A\cdot \zeta, \pm\frac{z}{c\zeta+d}\right]
\end{equation*}
then
\begin{equation*}
    h_{\tilde{\omega}}\circ \bar{h}=h_{\pm A}\circ h_{\tilde{\omega}}
\end{equation*}
since the semi-flat cscK metric $\eta$ on $\mathbb{H}^+\times_{Id}\mathbb{C}$ defined as (\ref{SF-cscK}) is $SL_2(\mathbb{Z})$-invariant, so
$\eta_X=h_{\tilde{\omega}}^*\eta$ is invariant by pull-back $\bar{h}$:
\begin{equation*}
    \bar{h}^*\eta_X= \bar{h}^*h_{\tilde{\omega}}^*\eta=
    h_{\tilde{\omega}}^*h_{\pm A}^*\eta=h_{\tilde{\omega}}^*\eta=\eta_X
\end{equation*}
So, if we inherit $\eta_X$ to $X^*$, denoted by $\underline{\eta_X}$,
it is also invariant by pull back of $h'^*$ due to $P\circ\bar{h}=h'\circ P$:
\begin{equation*}
    P^*h'^*\underline{\eta_X}=\bar{h}^*P^*\underline{\eta_X}=\bar{h}^*\eta_X=\eta_X=P^*\underline{\eta_X}
\end{equation*}
since $P^*$ is non-degenerate, so $h'^*\underline{\eta_X}=\underline{\eta_X}$
\end{proof}

\begin{theorem}\label{uniqueness}
    Let $\Phi:X\rightarrow \Sigma_g$ be a minimal elliptic surface equipped with a holomorphic section $s$, having non-constant Jacobian $J$, then it admits a $H(X,\Phi,s)$-invariant semi-flat cscK metric with fiber volume $1$ and scalar curvature $-3$. Furthermore, if $X$ possesses at least one singular fiber $X_p$
    with $J(p)\ne\infty$, then this metric is unique.
\end{theorem}

Before prove this, we need following lemmas:
\begin{lemma}\label{orbits-invariant}
    $\forall \beta\in \pi_1(\mathbb{H}_*^1, \xi_0)$ and $\gamma\in \pi_1(\Sigma_g^*,\sigma_0)$, their actions on $\mathbb{H}^+$ as deck transformations satisfy:
    \begin{equation*}
        \pi(\beta\cdot \gamma\cdot\zeta_0)=\pi (\beta\cdot \zeta_0)
    \end{equation*}
Consequently, there exists $\gamma'\in \pi_1(\Sigma_g^*,\sigma_0)$, such that
\begin{equation*}
    \beta\cdot \gamma\cdot \zeta=\gamma'\cdot \beta \cdot \zeta,
    \forall \zeta\in \mathbb{H}^+
\end{equation*}
\end{lemma}
\begin{proof}
    Let $\tilde{\gamma},\tilde{\beta}$ be lifts of $\gamma,\beta$
to $\mathbb{H^+}$ with $\tilde{\gamma}(0)=\tilde{\beta}(0)=\zeta_0$,
let $\alpha=\tilde{\omega}\circ \tilde{\gamma}$, lifts to $\tilde{\alpha}$ on $\mathbb{H}^+$ with $\tilde{\alpha}(0)=\tilde{\beta}(1)$, then:
\begin{equation*}
    \beta\cdot \gamma\cdot \zeta_0=\tilde{\alpha}(1)
\end{equation*}
$J\circ\pi\circ  \tilde{\beta}$ and $J\circ\pi\circ (\tilde{\beta}*\tilde{\alpha})$
are both loops on $\mathbb{P}^1_*$ based on $s_0=J(\sigma_0)$, furthermore, by $j\circ \tilde{\omega}=J\circ \pi$, we have:
\begin{equation*}
    \begin{split}
        J\circ \pi\circ \tilde{\beta}&=j\circ \tilde{\omega}\circ \tilde{\beta}=j\circ \beta\\
        J\circ\pi\circ (\tilde{\beta}*\tilde{\alpha})&=(j\circ \tilde{\omega}\circ \tilde{\beta})*(j\circ \tilde{\omega}\circ\tilde{\alpha})\\
        &=(j\circ \beta)*(j\circ \tilde{\omega}\circ \tilde{\gamma})\\
        &=(j\circ \beta)*(J\circ \gamma)
    \end{split}
\end{equation*}
since $J\circ \gamma\in \pi_1(\mathbb{P}^1_*,s_0)$ corresponds to the identity deck transformation of the covering $J:(\Sigma_g^*,\sigma_0)
\rightarrow (\mathbb{P}^1_*,s_0)$, we have:
\begin{equation*}
\begin{split}
    \pi(\beta\cdot\gamma\cdot \zeta_0)&=(J\circ\pi\circ (\tilde{\beta}*\tilde{\alpha}))\cdot \sigma_0\\
    &=((j\circ \beta)*(J\circ \gamma))\cdot \sigma_0\\
    &=(J\circ\gamma)\cdot ((j\circ\beta)\cdot \sigma_0)\\
    &=(j\circ \beta)\cdot \sigma_0\\
    &=(J\circ\pi\circ  \tilde{\beta})\cdot \sigma_0\\
    &=\pi(\beta\cdot \zeta_0)
\end{split}
\end{equation*}
as a result, there exists $\gamma'\in \pi_1(\Sigma_g^*,\sigma_0)$, such that
$\beta\cdot \gamma\cdot \zeta_0=\gamma'\cdot \beta\cdot \zeta_0$, since
$\beta\cdot \gamma$ and $\tilde{\gamma}\cdot \beta$ can be treated as deck transformation of $j\circ\tilde{\omega}=J\circ\pi$, so
\begin{equation*}
     \beta\cdot \gamma\cdot \zeta=\gamma'\cdot \beta \cdot \zeta,
    \forall \zeta\in \mathbb{H}^+
\end{equation*}
\end{proof}

\begin{lemma}\label{strongversion}
    If $\eta'$ is a semi-flat cscK metric of $\mathbb{H}^+_*\times_{Id}\mathbb{C}$
    with fiber volume $1$ and scalar curvature $-3$, which is also $\begin{pmatrix}
        1 & b\\
        0 & 1
    \end{pmatrix}$-invariant and $\begin{pmatrix}
        0 & 1\\
        -1 & 0
    \end{pmatrix}$-invariant, then
    \begin{equation*}
        \eta'=\eta=\sqrt{-1}\partial\bar{\partial}\left(
        \log \frac{1}{v^2}+\frac{y^2}{v}
        \right)
    \end{equation*}
where $(u+\sqrt{-1}v,x+\sqrt{-1}y)\in \mathbb{H}^+_*\times \mathbb{C}$.
Here, $\begin{pmatrix}
    0 & 1\\
    -1 & 0
\end{pmatrix}$ can be replaced as $\pm \begin{pmatrix}
    1 & 1\\
    -1 & 0
\end{pmatrix}$ or $\pm\begin{pmatrix}
    0 & -1\\
    1 & 1
\end{pmatrix}$.
\end{lemma}

\begin{proof}
    Let $V:=\{v_n\}_{n=1}^{\infty},v_n>v_{n+1}$ be a positive series such that
\begin{equation*}
    \mathbb{H}^+-\bigcup\limits_{n=1}^{\infty}\{\zeta\in\mathbb{H}^+|\text{Im}\zeta=v_n\}\subset \mathbb{H}^+_*
\end{equation*}
   following the proof of the theorem \ref{universal-rigidity}, there exists $f=f(u,v)>0$, such that:
    \begin{equation*}
        \eta'=\frac{\sqrt{-1}}{2}\left(\left(\frac{f(u,v)}{v^2}+\frac{y^2}{v^3}\right)d\zeta\wedge d\bar{\zeta}
    -\frac{y}{v^2}d\zeta\wedge d\bar{z}-\frac{y}{v^2}dz\wedge d\bar{\zeta}
    +\frac{1}{v}dz\wedge d\bar{z}
    \right)
    \end{equation*}
and $f_u(u,v)=0,\forall v\notin V$, the smoothness
of $f$ ensure that  $f_u(u,v)=0,\forall (u,v)\in \mathbb{H}^+_*$, therefore
\begin{equation}
    \frac{\partial f}{\partial u}(u,v)=f(v),\forall (u,v)
\in \mathbb{H}^+_*
\end{equation}
The scalar curvature condition reads:
\begin{equation*}
    \frac{d^2}{dv^2}\log f=3v^{-2}(f-1)
\end{equation*}
the $\begin{pmatrix}
        0 & 1\\
        -1 & 0
    \end{pmatrix}$-invariant of $\eta'$ shows
\begin{equation*}
    f(v)=f(v/(u^2+v^2)), \forall (u,v)\in \mathbb{H}^+_{*}
\end{equation*}
 which implies $f=constant=1$.
\end{proof}

\begin{proof}[Proof of the theorem \ref{uniqueness}]
The existence is guaranteed by construction \ref{etaX} together with proposition \ref{HXinvariant}. Now we turn to the uniqueness part.
   Let $\eta'$ be a given $H(X,\Phi,s)$-invariant semi-flat cscK metric with fiber volume $1$ and scalar curvature $-3$, let $\bar{\eta}':=P^*\eta'$,
   $\forall \beta\in \pi_1(H^+_*,\xi_0)$, let
   \begin{equation*}
   \begin{split}
        h^{\pm}_{\beta}:H^+\times_{\tilde{\omega}}\mathbb{C}&\rightarrow H^+\times_{\tilde{\omega}}\mathbb{C}\\
       [\zeta,z]&\rightarrow[\beta\cdot\zeta,\pm z]
   \end{split}
   \end{equation*}
we claim that $h_{\beta}^{\pm}$ can be projected to $\underline{h^{\pm}_{\beta}}\in Aut(X^*)$ such that:
\begin{equation*}
    \underline{h^{\pm}_{\beta}}\circ P=P\circ \underline{h^{\pm}_{\beta}}
\end{equation*}

\begin{equation*}
    \xymatrix{
   \ & H^+\times_{Id}\mathbb{C} &\ \\
   H^+\times_{\tilde{\omega}}\mathbb{C}\ar[ru]^{h_{\tilde{\omega}}}
   \ar[rr]^{h_{\beta}^{\pm}} \ar[d]_{P}&\ & H^+\times_{\tilde{\omega}}\mathbb{C}\ar[d]^{P}\ar[lu]_{h_{\tilde{\omega}}}\\
    X^*\ar@{-->}[rr]^{\underline{h_{\beta}^{\pm}}}\ar[d]_{\pi} & \ & X^*\ar[d]^{\pi}\\
    \Sigma_g^*\ar@{-->}[rr]^{\underline{\beta}} & \ & \Sigma_g^*
    }
\end{equation*}
to show this, we need to show $h^{\pm}_{\beta}([\gamma\cdot \zeta,\frac{z}{c_\gamma\tilde{\omega}(\zeta)+d_{\gamma}}])=[\beta\cdot\gamma\cdot\zeta,\frac{\pm z}{c_\gamma\tilde{\omega}(\zeta)+d_{\gamma}}]$ is congruent to $[\beta\cdot \zeta,\pm z]$, $\forall\gamma\in \pi_1(\Sigma_g^*,\sigma_0)$, where
\begin{equation*}
    \rho(\gamma)=\begin{pmatrix}
    a_{\gamma} & b_{\gamma}\\
    c_{\gamma} &d_{\gamma}
    \end{pmatrix}
\end{equation*}

 by the lemma \ref{orbits-invariant}, we can find $\gamma'\in \pi_1(\Sigma_g^*,\sigma_0)$, such that
\begin{equation*}
    \gamma'\cdot \beta\cdot \zeta=\beta\cdot \gamma\cdot \zeta
\end{equation*}
since
\begin{equation*}
    \tilde{\omega}(\gamma'\cdot \beta\cdot \zeta)=\tilde{\omega}(\beta\cdot \gamma\cdot \zeta)=\tilde{\omega}(\gamma\cdot \zeta)=\frac{a_{\gamma}\tilde{\omega}(\zeta)+b_{\gamma}}{c_{\gamma}\tilde{\omega}(\zeta)+d_{\gamma}}
    =\frac{a_{\gamma}\tilde{\omega}(\beta\cdot \zeta)+b_{\gamma}}{c_{\gamma}\tilde{\omega}(\beta\cdot \zeta)+d_{\gamma}}
\end{equation*}
we have $\rho(\gamma')=\pm \rho(\gamma)$, therefore:
\begin{equation*}
\begin{split}
        \gamma'\cdot [\beta\cdot \zeta,z]&=\left[\gamma'\cdot \beta\cdot \zeta,
    \frac{\pm z}{c_{\gamma}\tilde{\omega}(\beta\cdot \zeta)+d_{\gamma}}
    \right]\\
    &=\left[\beta\cdot\gamma\cdot\zeta,\frac{\pm z}{c_{\gamma}\tilde{\omega}(\zeta)+d_{\gamma}}\right]
\end{split}
\end{equation*}
as desired. The lemma \ref{orbits-invariant} also allows us to project $\beta$ to $\underline{\beta}\in Aut(\Sigma_g^*)$ with $\pi \circ \underline{h_{\beta}}
=\underline{\beta}\circ \pi$. Since $h_{\beta}$ preserve the zero section of $H^+\times_{\tilde{\omega}}\mathbb{C}$, the induced map $\underline{h_{\beta}}$ preserves
the section $s$. Furthermore, since $X$ is minimal with Kodaira dimension $\kappa(X)\ge 0$, we conclude $\underline{h_{\beta}}\in Aut(X)$, and the Riemann removable singularity theorem ensures $\underline{\beta}\in Aut(\Sigma_g)$, as a result:
\begin{equation*}
    (\underline{h_{\beta}},\underline{\beta})\in H(X,\Phi,s)
\end{equation*}
since $\eta'$ is $H(X,\Phi,s)$-invariant, so $\underline{h_{\beta}}^*\eta'=\eta'$, therefore $h_{\beta}^*(\bar{\eta}')=\bar{\eta}'$, and $\bar{\eta}'$ inherits to a $\rho(\pi_1(\Sigma_g^*))$-invariant semi-flat cscK metric $\tilde{\eta}'$ on
$H^+_*\times_{Id}\mathbb{C}$, such that
\begin{equation*}
    \bar{\eta}'=h_{\tilde{\omega}}^*(\tilde{\eta}')
\end{equation*}
since $X$ possesses at least one singular fiber $X_p$ with $J(p)\ne \infty$
so $\rho(\pi_1(\Sigma_g^*,\sigma_0))$ containing element $A$ conjugate with
$\begin{pmatrix}
    1 &\pm b\\
    0 & 1
\end{pmatrix}$ and element $B$ conjugate with one of
\begin{equation*}
    \left\{\pm \begin{pmatrix}
        0 & 1\\
        -1 & 0
    \end{pmatrix},\pm \begin{pmatrix}
        1 & 1\\
        -1 & 0
    \end{pmatrix},\pm \begin{pmatrix}
        0 & -1\\
        1 & 1
    \end{pmatrix}\right\}
\end{equation*}
and the lemma \ref{strongversion} implies $\tilde{\eta}'=\eta$ where
$\eta$ is given by (\ref{SF-cscK}). so $\eta'=\eta_X$.
\end{proof}


\section{Current extension}
In this section, we focus on the minimal elliptic surface
$\Phi:X\rightarrow \Sigma_g$ with globally defined holomorphic section,
and $\eta_X$ be the semi-flat cscK metric constructed in the section 2, and we try to extend $\eta_X$ as a closed positive $(1,1)$-current on the whole space $X$.
\subsection{Current extension on $X\setminus D_{\infty}$}
Let $D_{\infty}=\bigcup\limits_{J(p)=\infty}X_p$, and set $\check{X}:=X\setminus D_{\infty}$,  in this subsection, we will extend the semi-flat cscK metrci $\eta_X$ from
$X^*$ to $\check{X}$ as a positive closed $(1,1)$-current with vanishing Lelong number
everywhere.

Recall in the section 1.2.1, $\forall p\in \mathcal{P}':= \mathcal{P}\setminus J^{-1}(\infty)$, we have:
\begin{equation*}
    \Phi^{-1}(B^*_{\delta}(p))\cong (B^*_{\delta^{1/h_p}}(0)\times_{\omega_p}\mathbb{C})/\mathbb{Z}_{h_p}
\end{equation*}
Let $X^{orb}$ be the orbiford obtained by replacing $\Phi^{-1}(B_{\delta}(p))$ with $B_{\delta^{1/h_p}}(0)\times_{\omega_p}\mathbb{C}/\mathbb{Z}_{h_p}$, let $X'$ be the canonical reduction of $X^{orb}$ (see details in \cite{Ko63}), admitting a reduction map $R:X'\rightarrow \check{X}$, which obtianed by constracting divisors of $X'$ with self intersection number $-1$.

\begin{equation*}
    \xymatrix{
    X'\ar[rr]^{R'}\ar[d]_R \ar[rrd]^{\Phi'}&\ & X^{orb}\ar[d]^{\Phi^{orb}} \\
    \check{X}\ar[rr]^{\Phi} &\ & \Sigma_g\setminus J^{-1}(\infty)
    }
\end{equation*}

\begin{definition}
    Let $C$ be an regular irreducible curve of $X$, a closed (1,1)-current $\nu$ of $X$ is called $C^{\alpha}$-vanishing along $C$ ($\alpha>0$) if $\forall x\in C$, $\exists $ local coordinate $U_x,(s,t)$ with $C\cap U_x=\{s=0\}$ and a function
    $\phi_x\in C^{\infty}(U_x\setminus C)\cap C^0(U_x)$ such that
    on $U_x$:
    \begin{equation*}
    \begin{split}
        \phi_x&=O(|s|^{\alpha})\\
         \nu&=\sqrt{-1}\partial\bar{\partial}\phi_x
    \end{split}
    \end{equation*}
\end{definition}

\begin{lemma}\label{extension1}
     $\eta_X$ extend to a closed positive $(1,1)$-current
     on $X'$ which is $C^{\alpha}$-vaninshing or smooth along each irreducible components of $X'_p$, $\forall p\in \mathcal{P}'$.
\end{lemma}

\begin{proof}
Recall $\psi(\zeta,z)=-2\log v+y^2/v$ is the K\"ahler potential of $\eta$
on $\mathbb{H}^+\times\mathbb{C}$, and $\eta_X=h_{\omega_p}^*\eta$. Let $\psi_P:=h_{\omega_p}^*\psi$.

    \textbf{(1) The case $X_p$ is regular.} In this case, $X'_p=X_p,h_p=1$ and
\begin{equation*}
\begin{split}
        \omega_p(\tau)=\tau^{d_p}+\omega_0&,d_p\ge 2,\omega_0=u_0+\sqrt{-1}v_0,v_0>0\\
        \Phi^{-1}(B_{\delta}(0))&=B_{\delta}(0)\times_{\omega_p}\mathbb{C}
\end{split}
\end{equation*}
let $U_P=\{(\tau,z)| |\tau|<\epsilon_P,|z-z_0|<\delta_P,\epsilon_P,\delta_P<<1\}$, $P:(0,z_0)$, we have
\begin{equation*}
\begin{split}
    \phi_P(\tau,\bar{\tau},z,\bar{z}):=\psi_P=-2\log \omega_p(\tau) +\frac{(Im\ z)^2}{Im\ \omega_p(\tau)}\\
\end{split}
\end{equation*}
therefore $\Phi_P\in C^{\infty}(U_P)$, since $h_{\omega_P}$ holomorphic on $U_P\setminus X_P'$, we have that on $U_P\setminus X_P'$:
\begin{equation*}
    \sqrt{-1}\partial\bar{\partial}\phi_P=\sqrt{-1}\partial\bar{\partial}
    h_{\omega_P}^*\psi= h_{\omega_P}^*\sqrt{-1}\partial\bar{\partial}\psi=\eta_X
\end{equation*}
therefore: $\eta_X$ admits a smooth extension along $X_p$.

\textbf{(2) The case $X_p$ is of type $I_0^*$}. In this case $h_p=2$,
$\Phi'^{-1}(B_\delta(p))$ is the canonical reduction of
$B_{\delta^{1/2}}(0)\times_{\omega_p}\mathbb{C}/Z_2$, and
\begin{equation*}
    \begin{split}
\omega_p(\tau)&=\tau^{2d_p}+\omega_0,\omega_0=u_0+\sqrt{-1}v_0,v_0>0\\
        X'_p=X_p&=2\Theta+\Theta_1+\Theta_2+\Theta_3+\Theta_4\\
        \Theta_i\cdot\Theta&=o_i,(\Theta)^2=(\Theta_i)^2=-2\\
    \end{split}
\end{equation*}

There are four fixed points on $0\times_{\omega_0} \mathbb{C}$, say
$[0,0], [0,\omega_0/2],[0,1/2],[0,(\omega_0+1)/2]$, and $\Theta_i,1\le i\le 4$
are obtained by canonical reduction over these points. Let
\begin{equation*}
\begin{split}
     z_1=z,\ z_2=z-\frac{\omega_p(\tau)}{2},\  z_3&=z-\frac{1}{2}\ z_4=z-\frac{\omega_p(\tau)+1}{2}\\
     \alpha_i=\tau^2,\ \beta_i&=z_i^2,\ 1\le i\le 4\\
     p_i=\frac{\tau}{z_i},\ q_i&=\frac{z_i}{\tau},\ 1\le i\le 4
\end{split}
\end{equation*}
then the $Z_2$ action on $(\tau,z_i)$ is generated by:
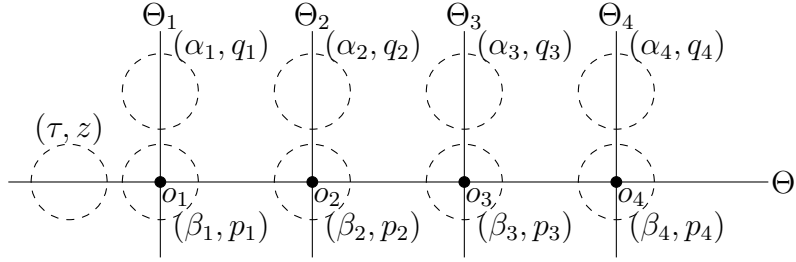
\begin{equation*}
    g:(\tau,z_i)\rightarrow(-\tau,-z_i)
\end{equation*}

\begin{equation*}
     \begin{tikzpicture}
\draw (-2,0)--(8,0); 
\draw (0,2)--(0,-1);   
\draw (2,2)--(2,-1);   
\draw (4,2)--(4,-1);   
\draw (6,2)--(6,-1);   
\draw[dashed] (-1.2,0) circle (0.5);
\draw[dashed] (0,0) circle (0.5);
\draw[dashed] (0,1.2) circle (0.5);
\draw[dashed] (2,0) circle (0.5);
\draw[dashed] (2,1.2) circle (0.5);
\draw[dashed] (4,0) circle (0.5);
\draw[dashed] (4,1.2) circle (0.5);
\draw[dashed] (6,0) circle (0.5);
\draw[dashed] (6,1.2) circle (0.5);
\filldraw[black] (0,0) circle (2pt); 
\filldraw[black] (2,0) circle (2pt);
\filldraw[black] (4,0) circle (2pt);
\filldraw[black] (6,0) circle (2pt);
\node at (0.2,-0.2) {$o_1$};
\node at (2.2,-0.2) {$o_2$};
\node at (4.2,-0.2) {$o_3$};
\node at (6.2,-0.2) {$o_4$};
\node at (-1.2,0.7) {$(\tau, z)$};
\node at (0.8,-0.6) {$(\beta_1, p_1)$};
\node at (2.8,-0.6) {$(\beta_2, p_2)$};
\node at (4.8,-0.6) {$(\beta_3, p_3)$};
\node at (6.8,-0.6) {$(\beta_4, p_4)$};
\node at (0.8,1.8) {$(\alpha_1, q_1)$};
\node at (2.8,1.8) {$(\alpha_2, q_2)$};
\node at (4.8,1.8) {$(\alpha_3, q_3)$};
\node at (6.8,1.8) {$(\alpha_4, q_4)$};
\node at (8.2,0) {$\Theta$};
\node at (0,2.2) {$\Theta_1$};
\node at (2,2.2) {$\Theta_2$};
\node at (4,2.2) {$\Theta_3$};
\node at (6,2.2) {$\Theta_4$};
\end{tikzpicture}
\end{equation*}

(2.1) $P\in \Theta-\{o_1,o_2,o_3,o_4\}$, this is similar as the case (1) and $\eta_X$  admits a smooth extension along $\Theta-\{o_1,o_2,o_3,o_4\}$

(2.2) $P=o_i$, we can choose $U_{P,i}=\{(\beta_i,p_i)|\  |\beta_i|+|p_i|<\epsilon_P<<1\}$
as local coordinate of $P=(0,0)$, we have:
\begin{equation*}
    \begin{split}
        \Phi(\beta_i,p_i)&=\beta_ip_i^2\\
        \Theta:\{p_i=0\}&,\ \Theta_1:\{\beta_i=0\}
    \end{split}
\end{equation*}
let
\begin{equation}\label{2.2.1}
    \phi_{P,i}:=\frac{1}{2}(g^*\psi_P+\psi_P)
\end{equation}
then $\phi_{P,i}$ is $Z_2$-invariant and inherit a well defined smooth function
on $U_{P,i}\setminus \{\beta_ip_i\ne 0\}$, and
\begin{equation*}
    \phi_{P,i}(\beta_i,\bar{\beta}_i,p_i,\bar{p}_i)=C_i+O(|\beta_i|^{1/2})
\end{equation*}
where
\begin{equation}\label{2.2.2}
    C_i=
    \left\{
\begin{aligned}
-2\log v_0, & \ i=1\\
-2\log v_0+\frac{3}{16v_0}, & \ i=2\\
-2\log v_0, & \ i=3\\
-2\log v_0+\frac{3}{16v_0}, &  \ i=4
\end{aligned}
\right.
\end{equation}

as a result $\phi_{P,i}\in C^0(U_{P,i})\cap C^{\infty}(U_{P,i}\setminus X_{p}')$. since $\eta_X$ is $Z_2$-invariant, so:
\begin{equation*}
    \eta_X=\sqrt{-1}\partial\bar{\partial}\phi_{P,i}\ \text{on}\ U_{P,i}\setminus X_p'
\end{equation*}
(2.3) $P\in \Theta_i\setminus\{o_i\}$, we can choose $U_{P,i}=\{(\alpha_i,q_i)|\ |\alpha_i|+|q_i-\hat{q}_i|<\epsilon_P<<1\}$ as local coordinate of $P=(0,\hat{q}_i)$,
we have:
\begin{equation*}
    \Phi(\alpha_i,q_i)=\alpha_i,\ \Theta_i:\{\alpha_i=0\}
\end{equation*}

Let $\phi_{P,i}$ defined as equation (\ref{2.2.1}), we have on
$U_{P,i}$:
\begin{equation*}
\begin{split}
        \phi_{P,i}(\alpha_i,\bar{\alpha}_i,q_i,\bar{q}_i)&=C_i+O(|\alpha_i|^{1/2})\\
        \eta_X&=\sqrt{-1}\partial\bar{\partial}\phi_{P,i}
\end{split}
\end{equation*}
where $C_i$ given as equation (\ref{2.2.2}).

as a summary of cases (2.1)-(2.3) we conclude that:
$\eta_X$ admits a $C^{1/2}$-vanishing extension along $\Theta_i,i=1,2,3,4$.

In the following figure, for convenience, a dashed line represents that $\eta_X$ can be smoothly extended to this component, while a label $C^\alpha$ annotated at the end of the line segment indicates that $\eta_X$ admits a $C^{\alpha}$-vanishing extension to this component:
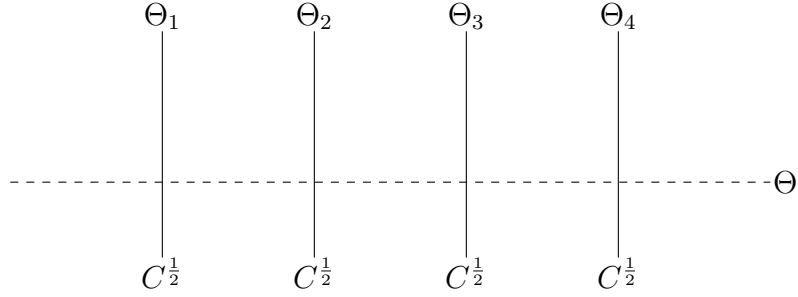
\begin{equation*}
     \begin{tikzpicture}
\draw[dashed](-2,0)--(8,0);
\draw (0,2)--(0,-1);
\draw (2,2)--(2,-1);
\draw (4,2)--(4,-1);
\draw (6,2)--(6,-1);
\node at (8.2,0) {$\Theta$};
\node at (0,2.2) {$\Theta_1$};
\node at (2,2.2) {$\Theta_2$};
\node at (4,2.2) {$\Theta_3$};
\node at (6,2.2) {$\Theta_4$};
\node at (0,-1.2) {$C^{\frac{1}{2}}$};
\node at (2,-1.2) {$C^{\frac{1}{2}}$};
\node at (4,-1.2) {$C^{\frac{1}{2}}$};
\node at (6,-1.2) {$C^{\frac{1}{2}}$};
\end{tikzpicture}
\end{equation*}

\textbf{(3) The case $X_p$ is of type $\mathrm{VI^*}$}. In this case $h_p=3$,
$\Phi^{-1}(B_{\delta}(p))$ is the canonical reduction of $(B_{\delta^{1/3}}(0)\times \mathbb{C})/Z_3$ and
\begin{equation*}
    \omega_p(\tau)=\frac{\eta-\eta^2\tau^{d_p}}{1-\tau^{d_p}},\ \eta=e^{2\pi i/3},J(p)=0,\ d_p\equiv 1\mod 3
\end{equation*}
\begin{equation*}
    \begin{split}
X_p'=X_p=3\Theta+2\Theta_{01}+\Theta_{02}&+2\Theta_{11}+\Theta_{12}+2\Theta_{21}+\Theta_{22}\\
        \Theta\cdot \Theta_{i1}&=o_i,i=0,1,2\\
        \Theta_{i2}\cdot \Theta_{i1}=c_i, (\Theta)^2&=(\Theta_{ij})^2=-2,i=0,1,2\ j=1,2
    \end{split}
\end{equation*}
There are three fixed points on $0\times_{\eta}\mathbb{C}$, say:
$\mathfrak{p}_0=[0,0],\mathfrak{p}_1=[0,\eta/3+2/3],\mathfrak{p}_2=[0,2\eta/3+1/3]$, and $\Theta_{ij},j=1,2$ are obtained by canonical reduction over $\mathfrak{p}_i$. Set:
\begin{equation*}
    \begin{split}
        z_0&=(1-\tau^{d_p})z\\
        z_1&=(1-\tau^{d_p})\left(z-\frac{\omega_p(\tau)}{3}-\frac{2}{3}\right)\\
        z_2&=(1-\tau^{d_p})\left(z-\frac{2\omega_p(\tau)}{3}-
        \frac{1}{3}\right)\\
        \alpha_i&=\tau^3,\beta_i=z_i^3,i=0,1,2\\
        p_{i1}&=\frac{\tau}{z_i^2},p_{i2}=\frac{\tau^2}{z_{i}},i=0,1,2\\
        q_{i1}&=\frac{z_i^2}{\tau},q_{i2}=\frac{z_i}{\tau^2},i=0,1,2
    \end{split}
\end{equation*}
the $Z_3$- action on $(\tau,z_i)$ is generated by
\begin{equation*}
    g\cdot (\tau,z_i)=(e_3\tau,e_3^{-1}z_i),e_3=e^{2\pi i/3}
\end{equation*}

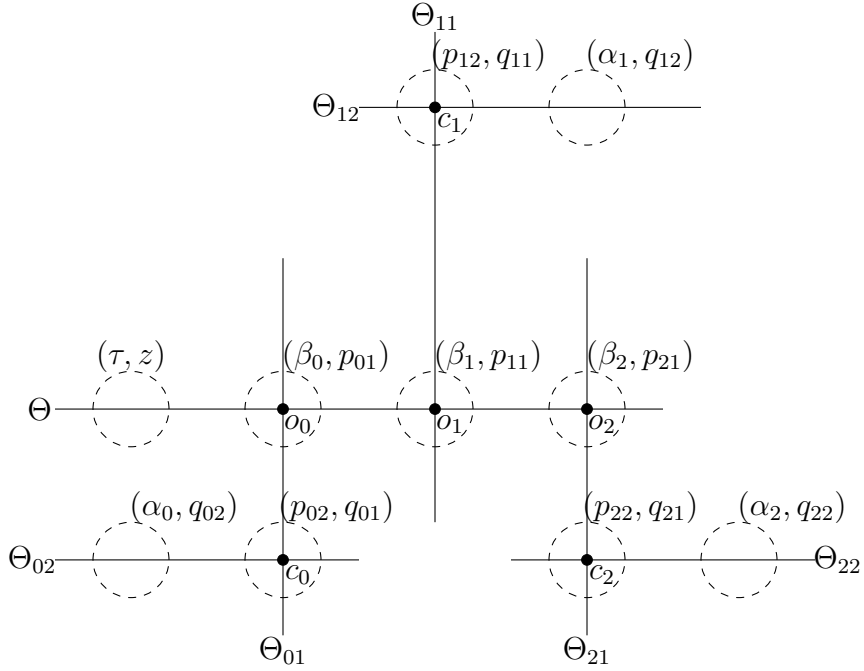
\begin{equation*}
    \begin{tikzpicture}
        \draw (-3,0)--(5,0);
        \draw (0,-3)--(0,2);
        \draw (2,-1.5)--(2,5);
        \draw (1,4)--(5.5,4);
        \draw (4,2)--(4,-3);
        \draw (-3,-2)--(1,-2);
        \draw (3,-2)--(7,-2);
        \draw[dashed] (0,0) circle (0.5);
        \draw[dashed] (2,4) circle (0.5);
        \draw[dashed] (4,4) circle (0.5);
        \draw[dashed] (0,-2) circle (0.5);
        \draw[dashed] (-2,-2) circle (0.5);
        \draw[dashed] (-2,0) circle (0.5);
        \draw[dashed] (2,0) circle (0.5);
        \draw[dashed] (4,0) circle (0.5);
        \draw[dashed] (4,-2) circle (0.5);
        \draw[dashed] (6,-2) circle (0.5);
        \filldraw[black] (0,0) circle (2pt);
        \filldraw[black] (2,0) circle (2pt);
        \filldraw[black] (4,0) circle (2pt);
        \filldraw[black] (0,-2) circle (2pt);
        \filldraw[black] (2,4) circle (2pt);
        \filldraw[black] (4,-2) circle (2pt);
        \node at (-3.2,0) {$\Theta$};
        \node at (-3.3,-2) {$\Theta_{02}$};
        \node at (0,-3.2) {$\Theta_{01}$};
        \node at (2,5.2) {$\Theta_{11}$};
        \node at (0.7,4) {$\Theta_{12}$};
        \node at (4,-3.2) {$\Theta_{21}$};
        \node at (7.3,-2) {$\Theta_{22}$};
        \node at (0.2,-0.2) {$o_0$};
        \node at (0.2,-2.2) {$c_0$};
        \node at (2.2,-0.2) {$o_1$};
        \node at (2.2,3.8) {$c_1$};
        \node at (4.2,-0.2) {$o_2$};
        \node at (4.2,-2.2) {$c_2$};
        \node at (-2,0.7) {$(\tau, z)$};
        \node at (0.7,0.7) {$(\beta_0, p_{01})$};
        \node at (0.7,-1.3) {$(p_{02}, q_{01})$};
        \node at (-1.3,-1.3) {$(\alpha_0, q_{02})$};
        \node at (2.7,0.7) {$(\beta_1, p_{11})$};
        \node at (2.7,4.7) {$(p_{12}, q_{11})$};
        \node at (4.7,4.7) {$(\alpha_1, q_{12})$};
        \node at (4.7,0.7) {$(\beta_2, p_{21})$};
        \node at (4.7,-1.3) {$(p_{22}, q_{21})$};
        \node at (6.7,-1.3) {$(\alpha_{2}, q_{22})$};
    \end{tikzpicture}
\end{equation*}

(3.1) $P\in \Theta-\{o_0,o_1,o_2\}$, this is similar as the case (1) and
$\eta_X$ extend smoothly on $\Theta-\{o_0,o_1,o_2\}$.

(3.2) $P\in \Theta_{i1}-\{c_i\}$, we can choose
$U_{P,i}=\{(\beta_i,p_{i1})|\ |\beta_i|+|p_{i1}-\hat{p}_{i1}|<\epsilon_{P}<<1\}$ as a local coordinate of $P=(0,\hat{p}_{i1})$:
\begin{equation*}
        \Phi(\beta_i,p_{i1})=\beta_i^2p_{i1}^3,\ \Theta_{i1}:\{\beta_i=0\}
\end{equation*}
let
\begin{equation}\label{case3.2}
    \phi_{P,i}:=\frac{1}{3}(g^{2*}\psi_P+g^*\psi_P+\psi_P)
\end{equation}
then $\phi_{P,i}$ is $Z_3$-invariant and inherits a well defined holomorphic function on $U_{P,i}$ where $\beta_ip_{i1}\ne 0$, and we have on $U_{P,i}$:
\begin{equation*}
\begin{split}
    \phi_{P,i}(\beta_i,\bar\beta{_i},p_{i1},\bar{p}_{i1})&=C_i+O(|\beta_i|^{1/3})\\
    \eta_X&=\sqrt{-1}\partial\bar{\partial}\phi_{P,i}
\end{split}
\end{equation*}
where $C_i$ are constants:
\begin{equation*}
    C_i=\left\{
\begin{aligned}
-\log\frac{3}{4}, & & i=0\\
-\log\frac{3}{4}+\frac{\sqrt{3}}{18}, & & i=1\\
-\log\frac{3}{4}+\frac{2\sqrt{3}}{9}, & & i=2\\
\end{aligned}
\right.
\end{equation*}
(3.3) $P=c_i$, we can choose $U_{P,i}'=\{(p_{i2},q_{i1})|\ |p_{i2}|+|q_{i1}|<\epsilon_P<<1\}$ as a local coordinate of
$P=(0,0)$, and
\begin{equation*}
    \Phi(p_{i2},q_{i1})=p_{i2}^2q_{i1},\ \Theta_{i1}:\{p_{i2}=0\},\ \Theta_{i2}:\{q_{i1}=0\}
\end{equation*}
let $\phi_{P,i}$ defined as equation (\ref{case3.2}), we have:
\begin{equation*}
    \phi_{P,i}(p_{i2},\bar{p}_{i2},q_{i1},\bar{q}_{i1})
    =C_i+O(|p_{i2}|^{1/3})=C_i+O(|q_{i1}|^{1/3})
\end{equation*}
(3.4) $P\in \Theta_{i2}-\{c_i\}$, we can choose
$U_{P,i}'=\{(\alpha_i,q_{i2}|\ |\alpha_i|+|p_{i2}-\hat{p}_{i2}|<\epsilon_P<<1)\}$ as local coordinate of $P=(0,\hat{p}_{i2})$ and
\begin{equation*}
    \Phi(\alpha_i,q_{i2})=\alpha_i,\Theta_{i2}:\{\alpha_i=0\}
\end{equation*}

let $\phi_{P,i}$ defined as equation (\ref{case3.2}), we have:
\begin{equation*}
    \phi_{P,i}(\alpha_i,\bar{\alpha}_i,q_{i2},\bar{q}_{i2})
    =C_i+O(|\alpha_i|^{1/3})
\end{equation*}
as a summary of (3.2)-(3.3),  we conclude that $\eta_X$ admits a
$C^{1/3}$-vanishing extension along $\Theta_{ij},i=0,1,2,j=1,2$.
\begin{equation*}
    \begin{tikzpicture}
        \draw[dashed] (-3,0)--(5,0);
        \draw (0,-3)--(0,2);
        \draw (2,-1.5)--(2,5);
        \draw (1,4)--(5.5,4);
        \draw (4,2)--(4,-3);
        \draw (-3,-2)--(1,-2);
        \draw (3,-2)--(7,-2);
        \node at (-3.2,0) {$\Theta$};
        \node at (-3.3,-2) {$\Theta_{02}$};
        \node at (1.1,-2.2) {$C^{1/3}$};
        \node at (0,-3.2) {$\Theta_{01}$};
        \node at (0,2.2) {$C^{1/3}$};
        \node at (2,5.2) {$\Theta_{11}$};
        \node at (2.4,-1.5) {$C^{1/3}$};
        \node at (0.7,4) {$\Theta_{12}$};
        \node at (6,4) {$C^{1/3}$};
        \node at (4,-3.2) {$\Theta_{21}$};
        \node at (4,2.2) {$C^{1/3}$};
        \node at (7.3,-2) {$\Theta_{22}$};
        \node at (3.1,-2.2) {$C^{1/3}$};
         \node at (0.2,-0.2) {$o_0$};
        \node at (0.2,-2.2) {$c_0$};
        \node at (2.2,-0.2) {$o_1$};
        \node at (2.2,3.8) {$c_1$};
        \node at (4.2,-0.2) {$o_2$};
        \node at (4.2,-2.2) {$c_2$};
    \end{tikzpicture}
\end{equation*}

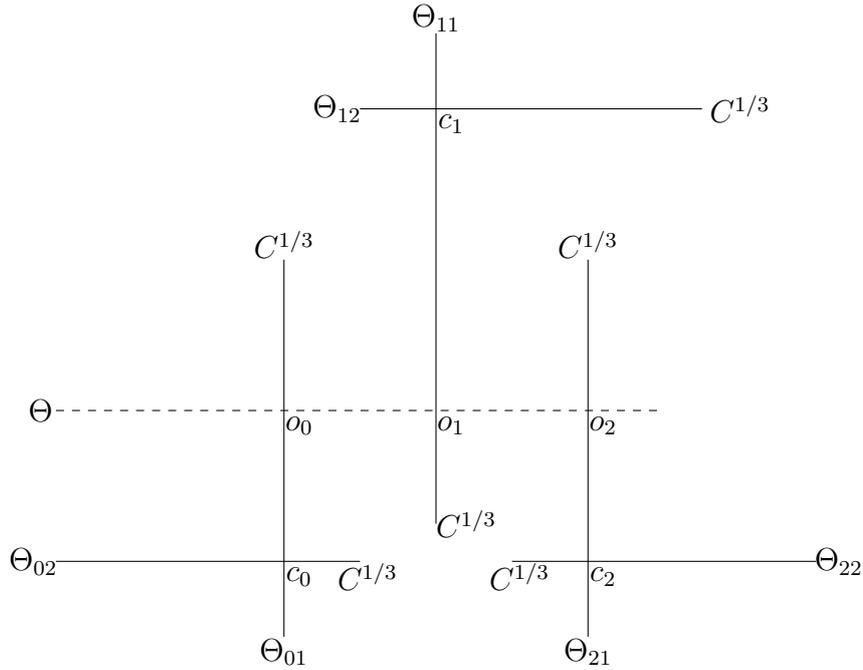
\textbf{(4) The case $X_p$ is of type $\mathrm{III^*}$}. In this case $h_p=4$,
$\Phi'^{-1}(B_{\delta}(p))$ is the canonical reduction of $B_{\delta^{1/4}}(0)\times_{\omega_p}\mathbb{C}/Z_4$ and
\begin{equation*}
    \begin{split}
        \omega_p(\tau)=\frac{\sqrt{-1}(1+\tau^{2d_p})}{1-\tau^{2d_p}}
        ,J(p)=1,d_p\equiv 1\mod 2
    \end{split}
\end{equation*}
\begin{equation*}
\begin{split}
      X_p'=X_p=4\Theta +3\Theta_{01}+2\Theta_{02}&+\Theta_{03}+
    3\Theta_{11}+2\Theta_{12}+\Theta_{13}+2\Theta_2\\
    \Theta\cdot \Theta_{i1}=\Theta_{i1}&=o_i,i=0,1\\
    \Theta\cdot \Theta_2&=o_2\\
    \Theta_{i1}\cdot\Theta_{i2}=c_i,\Theta_{i2}&\cdot\Theta_{i3}=e_i,i=0,1
\end{split}
\end{equation*}
there are theree fixed points on $0\times_{\sqrt{-1}}\mathbb{C}$, say
$\mathfrak{q}_0=[0,0],\mathfrak{q}_1=[0,\sqrt{-1}/2+1/2],\mathcal{q}_2=[0,1/2]$, and
$\Theta_{ij},i=0,1,j=1,2,3$ are obtained by canonical reduction over $\mathfrak{q}_i$,
and $\Theta_2$ is obtained by canonical reduction over $\mathfrak{q}_2$. set:
\begin{equation*}
    \begin{split}
        z_0&=(1-\tau^{2d_p})z\\
        z_1&=(1-\tau^{2d_p})\left(z-\frac{\omega_p(\tau)}{2}-\frac{1}{2}\right)\\
        z_2&=(1-\tau^{2d_p})\left(z-\frac{1}{2}\right)\\
        \alpha_i&=\tau^4,\beta_i=z_i^4,i=0,1\\
        \alpha_2&=\tau^2, \beta_2=z_2^2\\
        p_{ik}&=\frac{\tau^k}{z_i^{4-k}},i=0,1,k=1,2,3\\
        q_{ik}&=\frac{z_i^{4-k}}{\tau^k},i=0,1,k=1,2,3\\
        p_{2}&=\frac{\tau}{z_2},q_{2}=\frac{z_2}{\tau}
    \end{split}
\end{equation*}
then the $Z_4$-action on $(\tau,z_i)i=0,1$ is generated by
\begin{equation*}
    g\cdot (\tau,z_i)=(e_4\tau,e_4^{-1}z_i),e_4=e^{\pi i/2}
\end{equation*}
and the $Z_2$-action on $(\tau,z_2)$ is generated by
\begin{equation*}
    g^2\cdot (\tau,z_2)=(-\tau,-z_2)
\end{equation*}

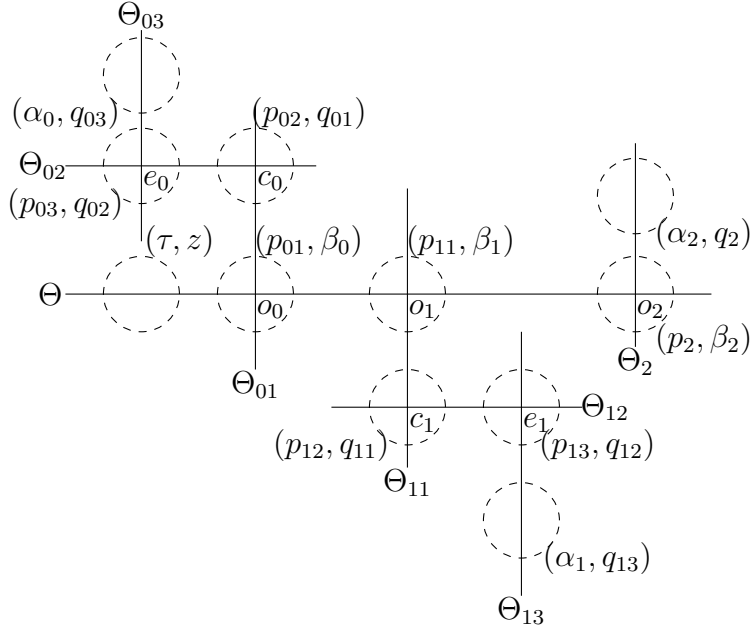
\begin{equation*}
\begin{tikzpicture}
    \draw (-3,0)--(5.5,0);
    \draw (-0.5,-1)--(-0.5,2.5);
    \draw (0.3,1.7)--(-3,1.7);
    \draw (-2,0.7)--(-2,3.5);
    \draw (1.5,1.4)--(1.5,-2.3);
     \draw (4.5,2)--(4.5,-0.7);
    \draw (0.5,-1.5)--(3.8,-1.5);
    \draw (3,-0.5)--(3,-4);
    \node at (-3.2,0) {$\Theta$};
    \node at (-3.3,1.7) {$\Theta_{02}$};
    \node at (-0.5,-1.2) {$\Theta_{01}$};
    \node at (-2,3.7) {$\Theta_{03}$};
    \node at (1.5,-2.5) {$\Theta_{11}$};
    \node at (4.1,-1.5) {$\Theta_{12}$};
    \node at (3,-4.2) {$\Theta_{13}$};
    \node at (4.5,-0.9) {$\Theta_{2}$};
     \draw[dashed] (-0.5,0) circle (0.5);
     \draw[dashed] (1.5,0) circle (0.5);
     \draw[dashed] (4.5,0) circle (0.5);
     \draw[dashed] (-2,0) circle (0.5);
     \draw[dashed] (1.5,-1.5) circle (0.5);
     \draw[dashed] (3,-1.5) circle (0.5);
     \draw[dashed] (-0.5,1.7) circle (0.5);
     \draw[dashed] (-2,1.7) circle (0.5);
     \draw[dashed] (-2,2.9) circle (0.5);
     \draw[dashed] (4.5,1.3) circle (0.5);
     \draw[dashed] (3,-3) circle (0.5);
     \node at (-0.3,-0.2) {$o_0$};
     \node at (1.7,-0.2) {$o_1$};
     \node at (4.7,-0.2) {$o_2$};
     \node at (-0.3,1.5) {$c_0$};
    \node at (-1.8,1.5) {$e_0$};
    \node at (1.7,-1.7) {$c_1$};
    \node at (3.2,-1.7) {$e_1$};
    \node at (-1.5,0.7) {$(\tau,z)$};
    \node at (0.2,0.7) {$(p_{01},\beta_0)$};
    \node at (0.2,2.4) {$(p_{02},q_{01})$};
    \node at (-3,1.2) {$(p_{03},q_{02})$};
    \node at (-3,2.4) {$(\alpha_0,q_{03})$};
    \node at (2.2,0.7) {$(p_{11},\beta_1)$};
    \node at (0.5,-2) {$(p_{12},q_{11})$};
    \node at (4,-2) {$(p_{13},q_{12})$};
    \node at (4,-3.5) {$(\alpha_1,q_{13})$};
    \node at (5.4,-0.6) {$(p_2,\beta_2)$};
    \node at (5.4,0.8) {$(\alpha_2,q_2)$};
\end{tikzpicture}
\end{equation*}
Applying a similar analysis as the case (3), we find:

(4.1) $\eta_X$ extends smoothly on $\Theta\setminus\{o_0,o_1,o_2\}$.

(4.2) $\eta_X$ admits a $C^{1/4}$-vanishing current extension along $\Theta_{i1},i=1,2$.

(4.3) $\eta_X$ admits a $C^{1/2}$-vanishing current extension along $\Theta_{i2},i=1,2$

(4.4) $\eta_X$ admits a $C^{1/2}$-vanishing current extension along $\Theta_{i3},i=1,2$.

(4.5) $\eta_X$ admits a $C^{1/2}$-vanishing current extension along $\Theta_2$.

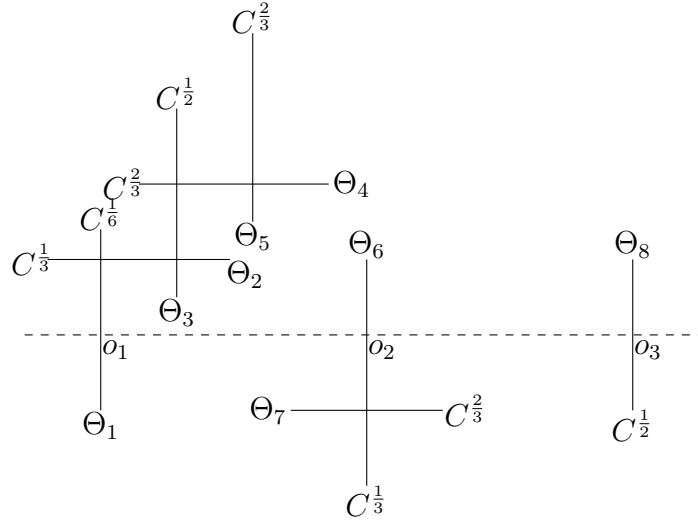
\begin{equation*}
\begin{tikzpicture}
    \draw[dashed] (-3,0)--(5.5,0);
    \draw (-0.5,-1)--(-0.5,2.5);
    \draw (0.3,1.7)--(-3,1.7);
    \draw (-2,0.7)--(-2,3.5);
    \draw (1.5,1.4)--(1.5,-2.3);
     \draw (4.5,2)--(4.5,-0.7);
    \draw (0.5,-1.5)--(3.8,-1.5);
    \draw (3,-0.5)--(3,-4);
    \node at (-3.2,0) {$\Theta$};
    \node at (-3.3,1.7) {$\Theta_{02}$};
    \node at (-0.5,-1.2) {$\Theta_{01}$};
    \node at (-2,3.7) {$\Theta_{03}$};
    \node at (1.5,-2.5) {$\Theta_{11}$};
    \node at (4.1,-1.5) {$\Theta_{12}$};
    \node at (3,-4.2) {$\Theta_{13}$};
    \node at (4.5,-0.9) {$\Theta_{2}$};
     \node at (-0.3,-0.2) {$o_0$};
     \node at (1.7,-0.2) {$o_1$};
     \node at (4.7,-0.2) {$o_2$};
     \node at (-0.3,1.5) {$c_0$};
    \node at (-1.8,1.5) {$e_0$};
    \node at (1.7,-1.7) {$c_1$};
    \node at (3.2,-1.7) {$e_1$};
    \node at (-0.5,2.7) {$C^{1/4}$};
    \node at (-2,0.5) {$C^{1/2}$};
    \node at (0.6,1.7) {$C^{1/2}$};
    \node at (1.5,1.7) {$C^{1/4}$};
    \node at (0.1,-1.5) {$C^{1/2}$};
    \node at (3,-0.3) {$C^{1/2}$};
    \node at (4.5,2.2) {$C^{1/2}$};
\end{tikzpicture}
\end{equation*}

\textbf{(5) The case $X_p$ is of type $II^*$}. In this case $h_p=6$,
$\Phi'^{-1}(B_{\delta}(p))$ is the canonical reduction of $B_{\delta^{1/6}}(0)\times_{\omega_p}\mathbb{C}/Z_6$ and
\begin{equation*}
    \begin{split}
        \omega_p(\tau)=\frac{\eta-\eta^2\tau^{2d_p}}{1-\tau^{2d_p}}
        ,J(p)=,d_p\equiv 2\mod 3
    \end{split}
\end{equation*}
\begin{equation*}
\begin{split}
      X_p'=X_p=6\Theta+5\Theta_1+4\Theta_2+3\Theta_3&+2\Theta_4+\Theta_5+
      4\Theta_6+2\Theta_7+3\Theta_8\\
    \Theta\cdot \Theta_1=o_1,\Theta\cdot\Theta_6&=o_2,\Theta\cdot\Theta_8=o_3\\
\Theta_1\cdot\Theta_2=\Theta_2\cdot\Theta_3=&\Theta_3\cdot\Theta_4=\Theta_4\cdot\Theta_5=1\\
   \Theta_6\cdot\Theta_7&=1
\end{split}
\end{equation*}
Applying similar analysis as above, we have:

(5.1) $\eta_X$ extends smoothly on $\Theta\setminus\{o_1,o_2,o_3\}$.

(5.2) $\eta_X$ admits a $C^{1/6}$-vanishing current extension along $\Theta_{1}$.

(5.3) $\eta_X$ admits a $C^{1/3}$-vanishing current extension along $\Theta_2$.

(5.4) $\eta_X$ admits a $C^{1/2}$-vanishing current extension along $\Theta_3$.

(5.5) $\eta_X$ admits a $C^{2/3}$-vanishing current extension along $\Theta_4$.

(5.6) $\eta_X$ admits a $C^{2/3}$-vanishing current extension along $\Theta_5$.

(5.7) $\eta_X$ admits a $C^{1/3}$-vanishing current extension along $\Theta_6$.

(5.8) $\eta_X$ admits a $C^{2/3}$-vanishing current extension along $\Theta_7$.

(5.9) $\eta_X$ admits a $C^{1/2}$-vanishing current extension along $\Theta_8$.

\begin{equation*}
    \begin{tikzpicture}
        \draw[dashed] (0,0)--(9,0);
        \draw (1,-1)--(1,1.4);
        \draw (0.3,1)--(2.7,1);
        \draw (2,0.5)--(2,3);
        \draw (1.5,2)--(4,2);
        \draw (3,1.5)--(3,4);
        \draw (4.5,1)--(4.5,-2);
        \draw (3.5,-1)--(5.5,-1);
        \draw (8,-1)--(8,1);
        \node at (1.2,-0.2) {$o_1$};
        \node at (4.7,-0.2) {$o_2$};
        \node at (8.2,-0.2) {$o_3$};
        \node at (1,1.6) {$C^{\frac{1}{6}}$};
        \node at (1.3,2) {$C^{\frac{2}{3}}$};
        \node at (0.1,1) {$C^{\frac{1}{3}}$};
        \node at (2,3.2) {$C^{\frac{1}{2}}$};
        \node at (3,4.2) {$C^{\frac{2}{3}}$};
        \node at (4.5,-2.2) {$C^{\frac{1}{3}}$};
        \node at (5.8,-1) {$C^{\frac{2}{3}}$};
        \node at (1,-1.2) {$\Theta_1$};
        \node at (2.9,0.8) {$\Theta_2$};
        \node at (2,0.3) {$\Theta_3$};
        \node at (4.3,2) {$\Theta_4$};
        \node at (3,1.3) {$\Theta_5$};
        \node at (4.5,1.2) {$\Theta_6$};
        \node at (3.2,-1) {$\Theta_7$};
        \node at (8,1.2) {$\Theta_8$};
        \node at (8,-1.2) {$C^{\frac{1}{2}}$};
    \end{tikzpicture}
\end{equation*}

\textbf{(6) The case $X_p$ is of type $IV$}. In this case $h_p=3$, $\Phi'^{-1}(B_{\delta}(p))$ is the canonical reduction of $B_{\delta^{1/3}}(0)\times_{\omega_p} \mathbb{C}/Z_3$ and
\begin{equation*}
    \omega_p(\tau)=\frac{\eta-\eta^2\tau^{d_p}}{1-\tau^{d_p}},J(p)=0,d_p\equiv 2\mod 3
\end{equation*}
\begin{equation*}
\begin{split}
     X_p'=3\Theta+\Theta_1+\Theta_2+\Theta_3\\
    \Theta\cdot \Theta_i=o_i,i=1,2,3, \Theta^2=-1\\
    X_p'\xrightarrow{contract\ \Theta} \Theta_1'+\Theta_2'+\Theta_3'=X_p
\end{split}
\end{equation*}
we have:

(1) $\eta_X$ extends smoothly to $\Theta\setminus\{o_1,o_2,o_3\}$

(2) $\eta_X$ admits a $C^{1/3}$-vanishing current extension along $\Theta_i$,i=1,2,3.

\begin{equation*}
    \begin{tikzpicture}
        \draw[dashed] (0,0)--(8,0);
        \draw (2,-1)--(2,1);
        \node at (2,1.2) {$C^{\frac{1}{3}}$};
        \node at (2,-1.2) {$\Theta_1$};
        \draw (4,-1)--(4,1);
        \node at (4,-1.2) {$\Theta_2$};
        \node at (4,1.2) {$C^{\frac{1}{3}}$};
        \draw (6,-1)--(6,1);
        \node at (6,-1.2) {$\Theta_3$};
        \node at (6,1.2) {$C^{\frac{1}{3}}$};
    \end{tikzpicture}
\end{equation*}

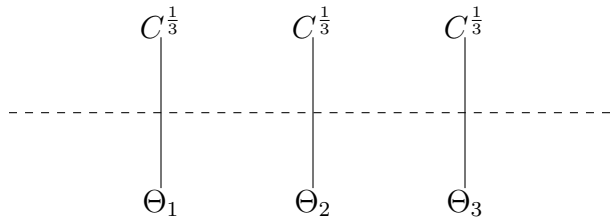
\textbf{(7) The case $X_p$ is of type $III$}. In this case $h_p=4$, $\Phi'^{-1}(B_{\delta}(p))$ is the canonical reduction of $B_{\delta^{1/3}(0)\times_{\omega_p}}\mathbb{C}/Z_4$, and
\begin{equation*}
    \begin{split}
        \omega_p(\tau)=\frac{\sqrt{-1}(1+\tau^{2d_p})}{1-\tau^{2d_p}},J(p)=1,d_p\equiv 1\mod2\\
        X_p'=4\Theta+\Theta_0+\Theta_1+2\Theta_2\\
        \Theta\cdot \Theta_i=o_i,i=0,1,2,\Theta^2=-1\\
        X_p'\xrightarrow{contract\ \Theta}\Theta_0'+\Theta_1'+2\Theta_2'\xrightarrow{contract\ \Theta_2'}
        \Theta_0''+\Theta_1''=X_p
    \end{split}
\end{equation*}
we have:

(1) $\eta_X$ extend smoothly to $\Theta\setminus\{o_0,o_1,o_2\}$

(2) $\eta_X$ admits a $C^{1/4}$-vanishing current extension along $\Theta_i$,i=0,1.

(3) $\eta_X$ admits a $C^{1/2}$-vanishing current extension along $\Theta_2$.
\begin{equation*}
    \begin{tikzpicture}
        \draw[dashed] (0,0)--(8,0);
        \draw (2,-1)--(2,1);
        \node at (2,1.2) {$C^{\frac{1}{4}}$};
        \node at (2,-1.2) {$\Theta_0$};
        \draw (4,-1)--(4,1);
        \node at (4,-1.2) {$\Theta_1$};
        \node at (4,1.2) {$C^{\frac{1}{4}}$};
        \draw (6,-1)--(6,1);
        \node at (6,-1.2) {$\Theta_2$};
        \node at (6,1.2) {$C^{\frac{1}{2}}$};
    \end{tikzpicture}
\end{equation*}

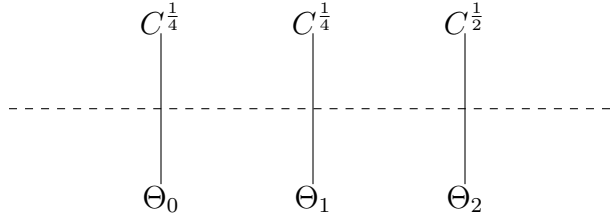
\textbf{(8) The case $X_p$ is of type $II$}. In this case $h_p=6$, $\Phi'^{-1}(B_{\delta}(p))$ is the canonical reduction of $B_{\delta^{1/6}}(0)\times_{\omega_p} \mathbb{C}/Z_6$ and
\begin{equation*}
    \omega_p(\tau)=\frac{\eta-\eta^2\tau^{2d_p}}{1-\tau^{2d_p}},J(p)=0,d_p\equiv 1\mod 3
\end{equation*}
\begin{equation*}
    \begin{split}
        X_p'=6\Theta+\Theta_1&+2\Theta_2+3\Theta_3\\
        \Theta\cdot \Theta_i=o_i,i=&1,2,3, \Theta^2=-1\\
        X_p'\xrightarrow{contract\ \Theta} \Theta_1'+2\Theta_2'+3\Theta_3'&
        \xrightarrow{contract\ \Theta_3'} \Theta_1''+2\Theta_2''
        \xrightarrow{contract\ \Theta_2''} \Theta_1'''=X_p
    \end{split}
\end{equation*}
we have:

(1) $\eta_X$ extend smoothly to $\Theta\setminus\{o_1,o_2,o_3\}$

(2) $\eta_X$ admits a $C^{1/6}$-vanishing current extension along $\Theta_1$

(3) $\eta_X$ admits a $C^{1/3}$-vanishing current extension along $\Theta_2$.

(4) $\eta_X$ admits a $C^{1/2}$-vanishing current extension along $\Theta_2$.
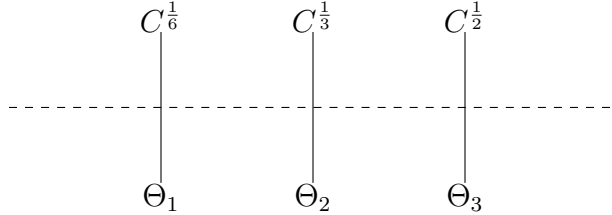
\begin{equation*}
    \begin{tikzpicture}
        \draw[dashed] (0,0)--(8,0);
        \draw (2,-1)--(2,1);
        \node at (2,1.2) {$C^{\frac{1}{6}}$};
        \node at (2,-1.2) {$\Theta_1$};
        \draw (4,-1)--(4,1);
        \node at (4,-1.2) {$\Theta_2$};
        \node at (4,1.2) {$C^{\frac{1}{3}}$};
        \draw (6,-1)--(6,1);
        \node at (6,-1.2) {$\Theta_3$};
        \node at (6,1.2) {$C^{\frac{1}{2}}$};
    \end{tikzpicture}
\end{equation*}
\end{proof}

\begin{corollary}\label{extension-first}
    $\eta_X$ extends a positive closed $(1,1)$-current on $\check{X}$
\end{corollary}
\begin{proof}
    Let $\eta_X'$ be the extension of $\eta_X$ to $X'$ as showed in the lemma
    \ref{extension1}, and $R_*\eta_X$ is our desired current extension of $\eta_X$
    on $\check{X}$.
\end{proof}

\subsection{Current extension on $X$}
In this subsection, we try to extend $\eta_X$ as a positive closed $(1,1)$-current on the whole $X$.

\begin{definition}(\textbf{double logarithmic pole})
    Let $X$ be a complex manifold and $C$ a subvariety of $X$. A (possibly singular) K\"ahler metric on $X$ is said to have a
    double logarithmic pole along
    $C$ if for every point $p\in C$, there exists a neighborhood $U(p)$ of $p$
    in $X$ and holomorphic functions $s_1,...,s_m$ defined on $U(p)$ such that:

    1. $C \cap U(p)$ is defined by the equation $s_1 s_2 \dots s_m = 0$ within $U(p)$.

    2. There exists a smooth real-valued function $f$ on $U(p)$ such that
    the restriction of $\omega_X$ on $U(p)-C$ can be expressed as:
    \begin{equation*} \omega_X |_{U(p)-C} = \sqrt{-1}\partial \bar{\partial} \left( -\sum_{i=1}^m c_i \log(\log|s_i|)^2 + f \right) \end{equation*}
    where $c_i$ are positive real numbers.
\end{definition}

Let $p$ be a pole of the Jacobian $J$ with order $b\ge 1$, then
\begin{equation*}
\begin{split}
      U_p^*:=\Phi^{-1}(B^*_{\delta}(p))=
    (\mathbb{H}_{\delta}^+\times \mathbb{C})/G_p =( \mathbb{H}_{\delta}^+\times_{\omega_b} \mathbb{C})/\mathbb{Z}
\end{split}
\end{equation*}
where $\omega_b(\zeta)=b\zeta$ $\tau=\Phi(\zeta)=e^{2\pi \sqrt{-1}\zeta}\in B^*_{\delta}(p)$.
and the metric $\eta_X$ restrict on $U_p^*$ as:
\begin{equation*}
    \eta_b=\sqrt{-1}\partial\bar{\partial}
    \left(-\log v^2+\frac{y^2}{bv}\right)
\end{equation*}
let
    \begin{equation*}
    \begin{split}
            \tau&=e^{2\pi \sqrt{-1} \zeta},  w=e^{2\pi \sqrt{-1}z}
    \end{split}
    \end{equation*}
Now we following Kodaira's construction
\cite{Ko63}:
\begin{equation*}
\begin{split}
     U_p^*=W':=B_{\delta}^*(0)\times \mathbb{C}^*/
    (\tau,w)\sim(\tau,w\tau^b)
\end{split}
\end{equation*}
where $\mathbb{C}^*:=\mathbb{C}-\{0\}$, and the metric
$\eta_b$ expressed in coordinate $(\tau,w)$ as:
\begin{equation}
    \eta_b=\sqrt{-1}\partial\bar{\partial}\left(
    -\log(\log|\tau|)^2-\frac{(\log|w|)^2}{2\pi b\log|\tau|}
    \right)
\end{equation}

When $b=1$, and $A_p=\begin{pmatrix}
    1 & 1\\
    0 & 1
\end{pmatrix}$, $X_p$ is of type $I_1$, let
\begin{equation*}
    W:=W'\cup (0\times \mathbb{C}^*)
\end{equation*}
set
\begin{equation*}
\begin{split}
    x(\tau,w)&=-2\sum_{n=1}^{\infty}(1-\tau^n)^{-2}\tau^n+
    \sum_{n=-\infty}^{+\infty}(1-w\tau^n)^{-2}w\tau^n\\
    y(\tau,w)&=\sum_{n=-\infty}^{+\infty}(1-w\tau^n)^{-3}(1+w\tau^n)w\tau^n
\end{split}
\end{equation*}
define
\begin{equation}
\begin{split}
      \kappa:W &\rightarrow  B_{\delta}(0)\times \mathbb{P}^2\\
    [\tau,w]&\rightarrow \tau\times [x(\tau,w):y(\tau,w):1]
\end{split}
\end{equation}
then $U_p:=\Phi^{-1}(B_{\delta}(p))$ is precisely
the subvariety of $B_{\delta}(0)\times \mathbb{P}^2$ determined by equation:
\begin{equation}
       F(\tau,x,y)= y^2-4x^3-x^2+g_2(\tau)x+g_3(\tau)=0
\end{equation}
where
\begin{equation*}
\begin{split}
        g_2(\tau)&=20\sum_{n=1}^{\infty}(1-\tau^n)^{-1}n^3\tau^n\\
        g_3(\tau)&=\frac{1}{3}\sum_{n=1}^{\infty}(1-\tau^n)^{-1}
        (7n^5+5n^3)\tau^n
\end{split}
\end{equation*}
and $\Phi$ is the projection $B_{\delta}(0)\times \mathbb{P}^2\rightarrow B_{\delta}(0)$, the singular fiber $X_p$
is the plane curve in $0\times\mathbb{P}^2$ defined by the equation:
\begin{equation*}
    y^2-4x^3-x^2=0
\end{equation*}
which has an ordinary double point at $\mathfrak{q}:x=y=0$, and $\kappa$ maps $W$ biholomorphically to $U_p-\mathfrak{q}$

If $q_0=(0,w_0)\in X_p-\mathfrak{q}=0\times\mathbb{C}^*
\subset W$, assume
\begin{equation*}
    \frac{1}{N}<|w_0|<N, N\in \mathbb{Z}^+
\end{equation*}
Let
\begin{equation*}
    U(q_0):=\{(\tau,w)|\ |\tau|<\frac{1}{N^2},
    \frac{1}{N}<|w|<N\}
\end{equation*}
then $\forall (\tau,w)\in U(q_0)$, $(\tau,w\tau^k)\in
U(q_0)$ if and only if $k=0$, so $U(q_0)$ can be chosen as a local coordinate near $q_0$, note that $X_p\cap U(q_0)=\{\tau=0\}\cap U(q_0)$, since
$\log |w|$ is bounded on $U(q_0)$, the following function extends smoothly to $\{\tau =0\}\cap U(q_0)$.
\begin{equation*}
    \frac{(\log|w|)^2}{\log|\tau|}
\end{equation*}
as a result, the semi-flat cscK metric have a double logarithmic pole along $X_p-\mathfrak{q}$. For general $b\ge 2$, $X_p=\Theta_1+...+\Theta_b$, $\Theta_i\cdot
\Theta_{i+1}=\mathfrak{q}_i$(we define $\Theta_{b+1}=\Theta_1$),
$\Phi^{-1}(B_{\delta}(p))-\{\mathfrak{q}_1,...,\mathfrak{q}_b\}$ is obtained by gluing $b$-copies of $W$, as a result, we have:

\begin{proposition}\label{doublelog1}
The semi-flat cscK metric $\eta_X$ have a double logarithmic pole along smooth locus of fiber of type $\mathrm{I}_b$.
\end{proposition}

When $A_p=\begin{pmatrix}
    -1 & -b\\
     0 & -1
\end{pmatrix}$
$X_p=\Theta_{00}+\Theta_{01}+\Theta_{10}+\Theta_{11}+
2\Theta_0+2\Theta_1+...+2\Theta_b$
is of type $\mathrm{I}_b^*$, where
\begin{equation}\label{intersection}
\begin{split}
        \Theta_{i}\cdot \Theta_{i+1}&=\mathfrak{q}_i,i=0,...,b-1\\
    \Theta_{0i}\cdot \Theta_0&=\mathfrak{q}_{0i},i=0,1\\
    \Theta_{1j}\cdot \Theta_b&=\mathfrak{q}_{1j},j=0,1
\end{split}
\end{equation}

in this case $\Phi^{-1}(B_\delta(p))-
\bigcup\limits_{i=0}^{b-1}\mathfrak{q}_{i}$ is obtained by gluing $2b$-copies $W_0,...,W_{2b-1}$ of $W$, then folding
this model by gluing $W_k$ with $W_{2b-k}$:
\begin{equation*}
    [\tau,w]_k \leftrightarrow [-\tau, (-1)^kw^{-1}],
    k=0,1...,2b-1
\end{equation*}
where we denote $[\tau,w]_k$ for points of $W_k$,
at last step, take canonical reduction at the fixed points
$[0,\pm 1]_0$ to obtain $\Theta_{0i},i=0,1$
and $[0,\pm \sqrt{-1}^b]_b$ to obtain
$\Theta_{1j},j=0,1$. To achieve the canonical reduction
at $[0,1]_0$, we choice local coordinate $(\tau,z)$
near $(0,0)$, where $e^{2\pi \sqrt{-1}z}=w$, then the folding step expressed as:
\begin{equation*}
    (\tau,z)\leftrightarrow (-\tau,-z)
\end{equation*}
set
\begin{equation*}
\alpha=\tau^2, \beta=z^2, p=\tau/z,q=z/\tau
\end{equation*}
consider the metric $\eta_b$ at coordinates $(\alpha,q),(\beta,p)$:
\begin{equation*}
\begin{split}
     \eta_{b}(\alpha,q)&=\sqrt{-1}\partial\bar{\partial}\left(
    -\log(\log|\alpha|)^2-\frac{4\pi y(\alpha,q)^2}{b\log|\alpha|}
    \right)\\
    y(\alpha,q)^2&=(\text{Im} z)^2<|\alpha q^2|\\
    \eta_{b}(\alpha,q)&=\sqrt{-1}\partial\bar{\partial}\left(
    -\log(\log|\beta p^2|)^2-\frac{4\pi y(\beta,p)^2}{b\log|\beta p^2|}\right)\\
    y(\beta,p)^2&<|\beta|
\end{split}
\end{equation*}
so $\eta_b$ have a
    double logarithmic pole along $\Theta_{00}-\{\mathfrak{q}_{00}\}$, similarly, this metric also have double logarithmic pole along
    $\Theta_{ij},i=0,1, j=0,1$
As a result, we obtain:
\begin{proposition}
    If $X_p=\Theta_{00}+\Theta_{01}+\Theta_{10}+\Theta_{11}+
2\Theta_0+2\Theta_1+...+2\Theta_b$ is of type $\mathrm{I}_b^*$ with
intersection relationship as (\ref{intersection}), then
the semi-flat cscK metric $\eta_X$ have a double logarithmic pole
along $X_p-\bigcup\limits_{i=0}^{b-1}\mathfrak{q}_i$.
\end{proposition}\label{doublelog2}
Let $S_{\infty}$ be the finite set containing singular points of $X_p$
with $J(p)=\infty$, then we have:
\begin{corollary}
    $\eta_X$ extends to $X\setminus S_{\infty}$ as a closed positive $(1,1)$-current.
\end{corollary}
\begin{proof}
    It sufficient to show the function $\log(\log |s|)^2$ is $L^1$ integrable
on $\Delta^2_{\epsilon}=\{(s,t)\in \mathbb{C}^2| |s|<\epsilon, |t|<\epsilon\}$,
and the desired existence of $\eta_X$ to $X\setminus S_{\infty}$ is then guaranteed by Corollary \ref{extension-first}, combined with Propositions \ref{doublelog1} and \ref{doublelog2}. Let $-\log|s|= r$, $dvol$ represent the standard Euclidean volume form of $\Delta_{\epsilon}^2$, we have:
\begin{equation*}
\begin{split}
      \int_{\Delta^2_{\epsilon}} \log(\log|s|^2) dvol&=2\pi \epsilon^2
    \int_{-\log\epsilon}^{+\infty}2(\log r )e^{-2r}dr\\
    &<2\pi\epsilon^2\int_{1}^{+\infty} e^{-r}dr\\
    &=\frac{2\pi\epsilon^2}{e}
\end{split}
\end{equation*}
\end{proof}
Finally, to extend $\eta_X$ on the whole $X$ as a closed positive $(1,1)$-current, we need following result due to Reese Harvey and John Polking\cite{HaPo75}:
\begin{theorem}[Harvey, Polking, 1975]\label{extensionfinite}
    Suppose $A$ is a closed subset of $\Omega$ open in $\mathbb{C}^n$. Let
    $p+k=n$. If $u$ is a $d$-closed, positive $p,p$ current on $\Omega-A$
    and $\Lambda_{2k-1}(A)=0$, then the trivial extension $\tilde{u}$ of $u$
    by zero across $A$ exists and $\tilde{u}$ is a $d$-closed, positive
    $p,p$ current on $\Omega$.
\end{theorem}
where $\Lambda_{2k-1}$ here denote the $2k-1$ dimensional Hausdorff measure on
$\mathbb{C}^n=\mathbb{R}^{2n}$.

In our case $A=S_{\infty}$ is a finite set, and $n=2,k=1$, so $\Delta_1(S_{\infty})=0$, as a result, we have:
\begin{corollary}
    $\eta_X$ extends to $X$ as a closed positive $(1,1)$-current.
\end{corollary}

\section{The associated class}
In this subsection, we try to compute the explicit class of the current $\eta_X$ which we obtained by extension of $\eta_X$ to $X$ in the section 4. In the decomposition (\ref{decomposition of H_2}):
\begin{definition}
    A nonzero class $\alpha\in H_{\mathbb{R}}^{1,1}(X)$ is called algebraic if $\alpha\in NS(X)\otimes \mathbb{R}$ and is called transcendental, if $\alpha\in (NS(X)\otimes \mathbb{Q})^{\perp}$
\end{definition}
A transcendental class $\alpha \in (NS(X)\otimes \mathbb{Q})^{\perp}$ can be represented
by a Poincar\'e 2-cycle of $X^*$, indeed we can obtain a series of generators of transcendental cycles using the monodromy of the fibration:

Step 1: Take $\gamma\in \pi_1(\Sigma_g^*,\sigma_0)$ such that $A_{\gamma}=\rho(\gamma)$ has an eigenvalue $1$, and choose a nonzero $\delta\in H_1(X_{_{\sigma_0}},\mathbb{Z})$ with $A_{\gamma}\cdot \delta=\delta$.

Step 2: Take a continuously varing of 1-cycles $\delta_t\in H_1(X_{\gamma(t)},\mathbb{Z}),\delta_0=\delta$ along $\gamma(t)$, then $\delta_1=\delta_0$ by our choice of $\delta$.

The union of these continuously varying 1-cycles $\cup\delta_t$, denoted by $C(\gamma,\delta)$ forms a 2-cycle of $X^*$. By construction, $C(\gamma,\delta)$ is disjoint from all singular fibers, so its intersection number with any irreducible component $\Theta_{pi}$ of a singular fiber $X_p$ is zero. Furthermore, the representative cycle for $\delta$ in $X_{\sigma_0}$ can be chosen to avoid the point $s(\sigma_0)$ which marked by the zero section $s$, this avoidance can be maintained during the continuous parallel transport along $\gamma$, as a result $[C(\gamma,\delta)]\cdot [s]=0$, so $[C(\gamma,\delta)]$ is transcendental.

\begin{lemma}\label{transcendental basis}
    For any $[C(\gamma,\delta)]$, we have
    \begin{equation*}
        \int_{[C(\gamma,\delta)]}\eta_X=0
    \end{equation*}
\end{lemma}
\begin{proof}
    since $A_{\gamma}=\rho([\gamma])$ has eigenvalue $1$, there exists $Q\in SL_2(\mathbb{Z}),n\in \mathbb{Z}$, such that
    \begin{equation*}
        A_{\gamma}=Q^{-1}\begin{pmatrix}
            1 & n\\
            0 & 1
        \end{pmatrix}Q,\ Q=\begin{pmatrix}
            a & b \\
            c & d
        \end{pmatrix}
    \end{equation*}

Let $\tilde{\gamma}(t),0\le t\le 1$ be a lift of $\gamma$ via the covering
$\pi:H^+\rightarrow \Sigma_g^*$, let $\zeta(t)=\tilde{\omega}(\tilde{\gamma}(t))$, then
\begin{equation*}
    \zeta(1)=A_{\gamma}\cdot \zeta_0
\end{equation*}
Let
\begin{equation*}
    \tilde{\delta}(t,s)=[\tilde{\gamma}(t),s\cdot (c\zeta(t)+d)]\in H^+\times_{\tilde{\omega}}\mathbb{C},0\le s \le 1
\end{equation*}
then
\begin{equation*}
    \tilde{\delta}(1,s)=\gamma\cdot \tilde{\delta}(0,s)
\end{equation*}
as a result
\begin{equation*}
    [\tilde{\delta}(1,s)]=[\tilde{\delta}(0,s)]\in X^*=H^+\times_{\tilde{\omega}}\mathbb{C}/\pi_1(\Sigma_g^*,\sigma_0)
\end{equation*}

and $\delta=[\tilde{\delta}(0,s)]$ is an eigenvector of the elliptic monodromy
around $\gamma$ w.r.t eigenvalue $1$. We have:
\begin{equation*}
     \begin{split}
\int_{[C(\gamma,\delta)]}\eta_X &=
\int_{[0,1]\times[0,1]}\tilde{\delta}^*\eta_X\\
  &= \int_{[0,1]\times[0,1]}\tilde{\delta}^*h_{\tilde{\omega}}^*\eta\\
  &=\int_{h_{\tilde{\omega}}\circ \tilde{\delta}([0,1]\times[0,1])}\eta
     \end{split}
\end{equation*}
where:
\begin{equation*}
    h_{\tilde{\omega}}\circ \tilde{\delta}(t,s)=[\zeta(t),s\cdot(c\zeta(t)+d)]
    \in H^+\times_{Id}\mathbb{C}
\end{equation*}
$Q$ induce the action $h_{Q}$ on $H^+\times_{Id}\mathbb{C}:$ $[\zeta,z]\rightarrow [Q\cdot \zeta,z/(c\zeta+d)]$, and we have:
\begin{equation*}
\begin{split}
      h_{Q}([\zeta(t),s\cdot(c\zeta(t)+d)])&=[Q\cdot \zeta(t),s]\\
    h_Q^*\eta&=\eta
\end{split}
\end{equation*}
Let $\tilde{\zeta}(t)=Q\cdot \zeta(t), \hat{\delta}(t,s)=[\tilde{\zeta}(t),s]$,  we have:
\begin{equation*}
\begin{split}
\tilde{\zeta}(1)&=\tilde{\zeta}(0)+n\\
     \int_{[C(\gamma,\delta)]}\eta_X&=\int_{\hat{\delta}([0,1]\times[0,1])}\eta\\
\end{split}
\end{equation*}
since
\begin{equation*}
\begin{split}
      \eta&=\sqrt{-1}\partial\bar{\partial}(-\log v^2+y^2/v)=d\phi\\
    where\ \phi&=
    \left(\frac{1}{v}+\frac{y^2}{2v^2}\right)d\bar{\zeta}-\frac{y}{v}d\bar{z}
\end{split}
\end{equation*}
The boundary of $\hat{\delta}([0,1]\times[0,1])$ are curves:
\begin{equation*}
    \begin{split}
        \gamma_1(t)&=[\tilde{\zeta}(t),0]\\
         \gamma_2(t)&=[\tilde{\zeta}(t),1]\\
         \gamma_3(s)&=[\tilde{\zeta}(0),s]\\
         \gamma_4(s)&=[\tilde{\zeta}(1),s]\\
         0\le t&,s\le 1
    \end{split}
\end{equation*}

\begin{equation*}
    \begin{split}
        \int_{\hat{\delta}([0,1]\times[0,1])}\eta&=
        \int_{\gamma_1}\phi-\int_{\gamma_2}\phi-\int_{\gamma_3}\phi+\int_{\gamma_4}\phi\\
        &=\int_{o}^1\frac{1}{Im\tilde{\zeta}(t)}d\bar{\tilde{\zeta}}(t)-\int_{o}^1\frac{1}{Im\tilde{\zeta}(t)}d\bar{\tilde{\zeta}}(t)\\
        &=0
    \end{split}
\end{equation*}
\end{proof}

\begin{corollary}
    $[\eta_X]\in H_{\mathbb{R}}^{1,1}(X)$ is algebraic
\end{corollary}
\begin{proof}
    This following from the lemma \ref{transcendental basis} and the fact that the transcendantal part of $H^{1,1}_{\mathbb{R}}(X)$ is spanned by $C(\gamma,\delta)$
we constructed preciously.
\end{proof}

Therefore:
\begin{equation*}
    [\eta_X]\in NS(X)\otimes \mathbb{R}
\end{equation*}
We can split $[\eta_X]$ as
\begin{equation*}
     [\eta_X]=T_X+N_X+E_X
\end{equation*}
where:
\begin{equation*}
    T_X=([\eta_X]\cdot [s],[\eta_X]\cdot [F])\cdot \begin{pmatrix}
        -\chi & 1\\
        1 & 0
    \end{pmatrix}^{-1}\cdot\begin{pmatrix}
        [s]\\
        [F]
    \end{pmatrix}
\end{equation*}

\begin{equation*}
    N_X=\sum_{p}N_p,\ \ N_p=([\eta_X]\cdot [\Theta_{p1}],...,[\eta_X]\cdot [\Theta_{p,m_{p-1}}])R_p^{-1}\begin{pmatrix}
        \Theta_{p1}\\
        \vdots\\
        \Theta_{p,m_{p}-1}
    \end{pmatrix}
\end{equation*}

\begin{equation*}
    E_X=([\eta_X]\cdot \psi(s_1),...,[\eta_X]\cdot\psi(s_r))\cdot E^{-1}\cdot \begin{pmatrix}
        \psi(s_1)\\
        \vdots\\
        \psi(s_r)
    \end{pmatrix}
\end{equation*}

(1) $T_X$: since $\eta_X$ has fiber volume $1$, so $[\eta_X]\cdot [F]=1$,
$[\eta_X]\cdot [s]=\int_{\Sigma_g^*}\omega^*\eta_{-1}=d\pi/3$, where
$\omega:\Sigma_g^*\rightarrow H^+$ is the multi-valued elliptic-period function with $j\circ \omega=J$ and $\eta_{-1}=\frac{du\wedge dv}{v^2}$ is the standard Poincar\'e metric with Gauss curvature $-1$, $d$ is the degree of the Jacobian $J$. We have:
\begin{equation*}
    T_X=[s]+\left(\frac{\pi d}{3}+\chi\right)[F]
\end{equation*}

(2) $N_X=\sum_pN_p$:

(2.1) $X_p$ is of type $\mathrm{III}$,
\begin{equation*}
\begin{split}
      X_p&=\Theta''_0+\Theta''_1\\
      [\Theta_0'']\cdot [s]&=1,
      [\Theta_0'']\cdot [\Theta_1'']=2, [\Theta_0'']^2=[\Theta_1'']^2=-2
\end{split}
\end{equation*}
which is obtained by reduction of $X_p'$ as following:
\begin{equation*}
\begin{split}
      \xymatrix{
\tilde{X}_p=4\Theta+\Theta_0+\Theta_1+2\Theta_2
\ar^{R_1}[d]
\ar@/_5pc/_{R}[dd]\\
\Theta_0'+\Theta_1'+2\Theta_2'
\ar^{R_2}[d]\\
\Theta_0''+\Theta_1''
    }\\
    R^{-1}(\Theta_0'')=R_1^{-1}(\Theta_0'+\Theta_2')
    =2\Theta+\Theta_0+\Theta_2\\
    R^{-1}(\Theta_1'')=R_1^{-1}(\Theta_1'+\Theta_2')
    =2\Theta+\Theta_1+\Theta_2
\end{split}
\end{equation*}
we have:
\begin{equation*}
    \begin{split}
        [\eta_X]\cdot [\Theta_1'']=[\eta_X']\cdot (2\Theta+\Theta_1+\Theta_2)=\frac{1}{2}
    \end{split}
\end{equation*}
so in this case:
\begin{equation*}
    N_p=-\frac{1}{4}\Theta_1''
\end{equation*}
(2.2) $X_p$ is of type $\mathrm{IV}$:
\begin{equation*}
    \begin{split}
        X_p&=\Theta_0'+\Theta_1'+\Theta_2',\
        [\Theta_0']\cdot [s]=1\\
    \end{split}
\end{equation*}
the intersection matrix of $\Theta_1',\Theta_2'$ is:
\begin{equation*}
\setlength{\arraycolsep}{1pt}
R_p=
    \begin{pmatrix}
        -2 & 1 \\
        1 & -2
    \end{pmatrix}
\end{equation*}
$X_p$ is obtianed by reduction of $X'_p$ as
following:
\begin{equation*}
    \begin{split}
        \xymatrix{
    \check{X}_p=3\Theta+\Theta_0+\Theta_1+\Theta_2
    \ar^{R}[d]\\
        \Theta_0'+\Theta_1'+\Theta_2'
        }\\
         R^{-1}(\Theta_i')=\Theta+\Theta_i,i=0,1,2
    \end{split}
\end{equation*}
we have:
\begin{equation*}
    [\eta_X]\cdot [\Theta_i']=[\eta_X']\cdot ([\Theta]+[\Theta_i'])
    =\frac{1}{3}
\end{equation*}
so:
\begin{equation*}
    N_p=(1/3,1/3)\cdot R_p^{-1}\cdot\begin{pmatrix}
        \Theta_1'\\
        \Theta_2'
    \end{pmatrix}
    =-\frac{1}{3}\Theta_1'-\frac{1}{3}\Theta_2'
\end{equation*}
(2.3) $X_p$ is of type $\mathrm{IV}^*$:
\begin{equation*}
X_p=\Theta_0+2\Theta_1+3\Theta_2+2\Theta_3+2\Theta_4
    +\Theta_5+\Theta_6,\ [\Theta_0]\cdot [s]=1
\end{equation*}
the intersection matrix of $\Theta_1,...,\Theta_6$ is:
\begin{equation*}
    \setlength{\arraycolsep}{1pt}
    R_p=
\begin{pmatrix}
        -2 & 1 &   &   &   &   \\
        1 & -2 & 1 & 1 &   &   \\
          & 1 & -2 &   & 1 &   \\
          & 1 &   & -2 &   & 1 \\
          &   & 1 &   & -2 &   \\
          &   &   & 1 &   & -2
    \end{pmatrix}
\end{equation*}
we have:
\begin{equation*}
    [\eta_X]\cdot(\Theta_1,...,\Theta_6)= (0,1/3,0,...,0)
\end{equation*}
so:
\begin{equation*}
\begin{split}
    N_p&=(0,1/3,0,...,0)\cdot R_p^{-1}\cdot \begin{pmatrix}
        \Theta_1\\
        \vdots\\
        \Theta_6
    \end{pmatrix}\\
   &=-\Theta_1-2\Theta_2-\frac{4}{3}
        \Theta_3-\frac{4}{3}\Theta_4-\frac{2}{3}\Theta_5-\frac{2}{3}\Theta_6
\end{split}
\end{equation*}
(2.4) $X_p$ is of type $\mathrm{\mathrm{III^*}}$:
\begin{equation*}
    X_p=\Theta_0+2\Theta_1+3\Theta_2+
4\Theta_3+3\Theta_4+2\Theta_5+2\Theta_6
+\Theta_7,\ [\Theta_0]\cdot [s]=1
\end{equation*}
the intersection matrix of $\Theta_1,...,\Theta_7$ is:
\begin{equation*}
    \setlength{\arraycolsep}{1pt}
    R_p=\begin{pmatrix}
-2&1& & & & & \\
1&-2&1& & & & \\
 &1&-2&1&1& & \\
 & &1&-2& &1& \\
 & &1& &-2& & \\
 & & &1& &-2&1\\
 & & & & &1&-2
\end{pmatrix}
\end{equation*}
we have:
\begin{equation*}
    [\eta_X]\cdot (\Theta_1,...,\Theta_7)=(0,0,1/4,0,...,0)
\end{equation*}
so:
\begin{equation*}
\begin{split}
     N_p&=(0,0,1/4,0,...,0)\cdot R_p^{-1}\cdot \begin{pmatrix}
        \Theta_1\\
        \vdots\\
        \Theta_7
    \end{pmatrix}\\
    &=-\Theta_1-2\Theta_2-3\Theta_3-\frac{9}{4}\Theta_4-\frac{3}{2}\Theta_5-\frac{3}{2}\Theta_6-\frac{3}{4}\Theta_7
\end{split}
\end{equation*}
(2.5) $X_p$ is of type $\mathrm{II}^*$:
\begin{equation*}
\begin{split}
X_p=\Theta_0+2\Theta_1+3\Theta_2+4\Theta_3+5\Theta_4&+6\Theta_5+4\Theta_6+3\Theta_7+2\Theta_8\\
\ [\Theta_0]\cdot [s]&=1
\end{split}
\end{equation*}
the intersection matrix of $\Theta_1,...,\Theta_8$ is:
\begin{equation*}
    \setlength{\arraycolsep}{1pt}
    R_p=\begin{pmatrix}
-2&1& & & & & & \\
1&-2&1& & & & & \\
 &1&-2&1& & & & \\
 & &1&-2&1& & & \\
 & & &1&-2&1&1& \\
 & & & &1&-2& &1\\
 & & & &1& &-2& \\
 & & & & &1& &-2
\end{pmatrix}
\end{equation*}
we have:
\begin{equation*}
    [\eta_X]\cdot (\Theta_1,...,\Theta_8)=(0,0,0,0,1/6,0,0,0)
\end{equation*}
so:
\begin{equation*}
\begin{split}
     N_p&=(0,0,0,0,1/6,0,0,0)\cdot R_p^{-1}\cdot
     \begin{pmatrix}
        \Theta_1\\
        \vdots\\
        \Theta_8
    \end{pmatrix}\\
    &=-\Theta_1-2\Theta_2-3\Theta_3-4\Theta_4-5\Theta_5-\frac{10}{3}\Theta_6-\frac{5}{2}\Theta_7-\frac{5}{3}\Theta_8
\end{split}
\end{equation*}
(2.5) $X_p$ is of type $\mathrm{I_0^*}$:
\begin{equation*}
     X_p=\Theta_0+\Theta_1+\Theta_2+\Theta_3+2\Theta_4,\
    [\Theta_0]\cdot [s]=1
\end{equation*}
the intersection matrix of $\Theta_1,...,\Theta_4$ is:
\begin{equation*}
     \setlength{\arraycolsep}{1pt}
    R_p= \begin{pmatrix}
-2& & &1\\
 &-2& &1\\
 & &-2&1\\
1&1&1&-2
\end{pmatrix}
\end{equation*}
we have:
\begin{equation*}
    [\eta_X]\cdot (\Theta_1,...,\Theta_4)=(0,0,0,1/2)
\end{equation*}
so:
\begin{equation*}
\begin{split}
     N_p&=(0,0,0,1/2)\cdot R_p^{-1}\cdot \begin{pmatrix}
        \Theta_1\\
        \Theta_2\\
        \Theta_3\\
        \Theta_4
    \end{pmatrix}\\
    &=-\frac{1}{2}\Theta_1-\frac{1}{2}\Theta_2-\frac{1}{2}\Theta_3-\Theta_4
\end{split}
\end{equation*}
(2.6) $X_p$ is of type $\mathrm{I}_b^*,b\ge 1$:
\begin{equation*}
X_p=\Theta_0+\Theta_1+\Theta_2+\Theta_3+2\Theta_4+...+2\Theta_{4+b},\ [\Theta_0]\cdot [s]=1
\end{equation*}
the intersection matrix of $\Theta_1,...,\Theta_{4+b}$ is:
\begin{equation*}
R_p=
\setlength{\arraycolsep}{1pt} 
\left(
\begin{array}{@{}ccc|cccc@{}}
  -2     &        &        & 1      & 0      & \cdots & 0      \\ 
         & -2     &        & 0      &        & \cdots & 1      \\ 
         &        & -2     & 0      &        & \cdots & 1      \\ 
  \hline
  1      & 0      & 0      & -2     & 1      &        &        \\ 
  0      &        & 0      & 1      & -2     & 1      &        \\ 
  \vdots & \vdots & \vdots &        & 1      & -2     & \ddots \\
  0      & 1      & 1      &        &        & \ddots & -2     \\ 
\end{array}
\right)
\end{equation*}
we have:
\begin{equation*}
    [\eta_X]\cdot (\Theta_1,...,\Theta_{4+b})
    =\left(0,0,0,\frac{1}{4b},\frac{1}{2b},...,\frac{1}{2b},\frac{1}{4b}\right)
\end{equation*}
so:
\begin{equation*}
\begin{split}
     N_p&=\left(0,0,0,\frac{1}{4b},\frac{1}{2b},...,\frac{1}{2b},\frac{1}{4b}\right)
\cdot R_p^{-1}\cdot \begin{pmatrix}
        \Theta_1\\
        \vdots\\
        \Theta_{4+b}
    \end{pmatrix}\\
    &=-\frac{1}{2}\Theta_1-\frac{b+4}{8}\Theta_2-\frac{b+4}{8}\Theta_3-\sum\limits_{n=0}^{b}\left(1+\frac{n(2b-n)}{4b}\right)\Theta_{4+n}
\end{split}
\end{equation*}
(2.7) $X_p$ is of type $\mathrm{I}_b,b\ge 2$:
\begin{equation*}
    X_p=\Theta_0+\Theta_1+...+\Theta_{b-1},\ [\Theta_0]\cdot [s]=1
\end{equation*}
the intersection matrix of $\Theta_1,...,\Theta_{b-1}$ is:
\begin{equation*}
    \setlength{\arraycolsep}{1pt}
    R_p=
\begin{pmatrix}
-2 & 1 & & & & \\
1 & -2 & 1 & & & \\
& 1 & -2 & 1 & & \\
& & \ddots & \ddots & \ddots & \\
& & & 1 & -2 & 1 \\
& & & & 1 & -2
\end{pmatrix}
\end{equation*}
we have:
\begin{equation*}
    [\eta_X]\cdot (\Theta_1,...,\Theta_n)=(1/b,...,1/b)
\end{equation*}

so:
\begin{equation*}
\begin{split}
    N_p&=(1/b,...,1/b)\cdot R_p^{-1}\cdot \begin{pmatrix}
        \Theta_1\\
        \vdots\\
        \Theta_{b}
    \end{pmatrix}\\
    &=-\sum\limits_{i=1}^{b-1}\frac{i(b+1-i)}{2b}\Theta_{i}
\end{split}
\end{equation*}

In summary, we have the following:
\begin{proposition}\label{class}
    The associated class $[\eta_X]\in H^{1,1}_{\mathbb{R}}(X)$ can be computed as following:
    \begin{equation*}
        [\eta_X]=[s]+\left(\frac{\pi d}{3}+\chi\right)[F]+\sum_{p}N_p+E_X
    \end{equation*}
\end{proposition}
where:
\begin{equation*}
    E_X=([\eta_X]\cdot \psi(s_1),...,[\eta_X]\cdot\psi(s_r))\cdot E^{-1}\cdot \begin{pmatrix}
        \psi(s_1)\\
        \vdots\\
        \psi(s_r)
    \end{pmatrix}
\end{equation*}
and $N_p$ listed as the table \ref{listofNX}.
\begin{table}[]
\centering
    \caption{List of $N_p$}
    \label{listofNX}
    \begin{tabular}{@{} >{\centering\arraybackslash}m{1cm} >{\centering\arraybackslash}m{5cm} >{\centering\arraybackslash}m{5cm} @{}}
    \toprule
    Fiber type & Dynkin diagram & $N_p$ \\
    \midrule
        $\mathrm{I_0^*}$ &
           \begin{tikzpicture}[scale=0.8]
              \node (T0) at (-1.5,0.5) {$\Theta_0$};
              \node (T3) at (1.5,-0.5) {$\Theta_3$};
              \node (T2) at (1.5,0.5) {$\Theta_2$};
              \node (T1) at (-1.5,-0.5) {$\Theta_1$};
              \node (T4) at (0,0) {$\Theta_4$};
              \draw[dashed]  (T0.east)--(T4.west);
              \draw  (T1.east)--(T4.west);
              \draw  (T2.west)--(T4.east);
              \draw  (T3.west)--(T4.east);
           \end{tikzpicture}&
           $-\frac{1}{2}\Theta_1-\frac{1}{2}\Theta_2-\frac{1}{2}\Theta_3-\Theta_4$\\
   \midrule
        $$\mathrm{I}_b^*,b\ge 1$$&

            \begin{tikzpicture}[scale=0.8]
              \node (T0) at (-1.5,0.5) {$\Theta_0$};
              \node (T3) at (3.5,-0.5) {$\Theta_3$};
              \node (T2) at (3.5,0.5) {$\Theta_2$};
              \node (T1) at (-1.5,-0.5) {$\Theta_1$};
              \node (T4) at (0,0) {$\Theta_4$};
              \node (T5) at (2,0) {$\Theta_{b+4}$};
              \draw[dashed]  (T0.east)--(T4.west);
              \draw  (T1.east)--(T4.west);
              \draw  (T4.east)--(T5.west);
              \draw  (T5.east)--(T2.west);
              \draw  (T5.east)--(T3.west);
              \end{tikzpicture} &
                              $ -\frac{1}{2}\Theta_1-\frac{b+4}{8}\Theta_2
                              -\frac{b+4}{8}\Theta_3-\sum\limits_{n=0}^{b}
                              \left(
                              1+\frac{n(2b-1)}{4b}\Theta_{4+n}
                              \right)
                              $ \\
\midrule
$$\mathrm{I}_b,b\ge 2$$ &
    \begin{tikzpicture}[scale=0.8]
    \node (T0) at (-0,0.6) {$\Theta_0$};
    \node (T1) at (-2,0) {$\Theta_1$};
    \node (T2) at (2,0) {$\Theta_{b-1}$};
    \draw [dashed] (T0.west)--(T1.east);
   \draw [dashed] (T0.east)--(T2.west);
   \draw  (T1.east)--(-1,0);
   \draw  (T2.west)--(1,0);
   \draw [dashed]  (-1,0)--(1,0);
\end{tikzpicture}&
$$
-\sum\limits_{i=1}^{b-1}\frac{i(b+1-i)}{2b}\Theta_i
$$\\
\midrule
$$\mathrm{II^*}$$ &
    \begin{tikzpicture}[rotate=-90][scale=0.8]
        \node (T0) at (-1,0.3) {$\Theta_0$};
        \node (T1) at (-1,-0.7) {$\Theta_1$};
        \node (T2) at (0,-0.7) {$\Theta_2$};
        \node (T3) at (0,0.3) {$\Theta_3$};
        \node (T4) at (0,1.3) {$\Theta_4$};
        \node (T5) at (0,2.3) {$\Theta_5$};
        \node (T7) at (-1,2.3){$\Theta_7$};
        \node (T6) at (0,3.3) {$\Theta_6$};
        \node (T8) at (-1,3.3) {$\Theta_8$};
        \draw (T1.south)--(T2.north);
        \draw (T2.east)--(T3.west);
        \draw (T3.east)--(T4.west);
        \draw (T4.east)--(T5.west);
        \draw (T5.east)--(T6.west);
        \draw (T6.north)--(T8.south);
        \draw[dashed] (T0.west)--(T1.east);
        \draw (T5.north)--(T7.south);
    \end{tikzpicture} &
    $-\Theta_1-2\Theta_2-3\Theta_3-4\Theta_4-5\Theta_5-\frac{10}{3}\Theta_6-\frac{5}{2}\Theta_7-\frac{5}{3}\Theta_8$\\
\midrule
$$\mathrm{III^*}$$ &
\begin{tikzpicture}[rotate=-90][scale=0.8]
        \node (T0) at (-1,-1) {$\Theta_0$};
        \node (T1) at (0,-1) {$\Theta_1$};
        \node (T2) at (0,0) {$\Theta_2$};
        \node (T3) at (0,1) {$\Theta_3$};
        \node (T5) at (-1,1) {$\Theta_5$};
        \node (T4) at (0,2) {$\Theta_4$};
        \node (T6) at (0,3) {$\Theta_6$};
        \node (T7) at (-1,3) {$\Theta_7$};
        \draw [dashed]  (T0.south)--(T1.north);
        \draw  (T1.east)--(T2.west);
        \draw (T2.east)--(T3.west);
        \draw  (T3.east)--(T4.west);
        \draw  (T4.east)--(T6.west);
        \draw  (T6.north)--(T7.south);
        \draw  (T3.north)--(T5.south);
    \end{tikzpicture} &
    $-\Theta_1-2\Theta_2-3\Theta_3-\frac{9}{4}\Theta_4-\frac{3}{2}\Theta_5-\frac{3}{2}\Theta_6-\frac{3}{4}\Theta_7$ \\
\midrule
$$
\mathrm{IV^*}
$$ &
\begin{tikzpicture}[scale=0.8]
    \node (N2) at (0,0) {$\Theta_2$};

    \node (N1) at (-1,0) {$\Theta_1$};
    \node (N0) at (-2,0) {$\Theta_0$};

    \node (N4) at (1,0) {$\Theta_4$};
    \node (N6) at (2,0) {$\Theta_6$};

    \node (N3) at (0,1) {$\Theta_3$};
    \node (N5) at (0,2) {$\Theta_5$};

    \draw[dashed] (N0) -- (N1); 
    \draw (N1) -- (N2);         
    \draw (N2) -- (N4);         
    \draw (N4) -- (N6);         
    \draw (N2) -- (N3);         
    \draw (N3) -- (N5);         
\end{tikzpicture} &
$-\Theta_1-2\Theta_2-\frac{4}{3}
        \Theta_3-\frac{4}{3}\Theta_4-\frac{2}{3}\Theta_5-\frac{2}{3}\Theta_6$\\
\midrule
$$
\mathrm{III}
$$ &
\begin{tikzpicture}
    \node (T0) at (-1,0) {$\Theta_0$};
    \node (T1) at (1,0) {$\Theta_1$};
    \draw [dashed]  (-0.8,0.1)--(0.8,0.1);
    \draw [dashed]  (-0.8,-0.1)--(0.8,-0.1);
\end{tikzpicture} &
$$-\frac{1}{4}\Theta_1$$ \\
\midrule
$$
\mathrm{IV}
$$ &
\begin{tikzpicture}
\node (T0) at (0,0.3) {$\Theta_0$};
    \node (T1) at (-1,-0.5) {$\Theta_1$};
    \node (T2) at (1,-0.5) {$\Theta_2$};
    \draw[dashed]  (T0.south west)--(T1.north east);
    \draw[dashed]  (T0.south east)--(T2.north west);
    \draw  (T1.east)--(T2.west);
\end{tikzpicture} &
$-\frac{1}{3}\Theta_1-\frac{1}{3}\Theta_2$\\
\bottomrule
    \end{tabular}
\end{table}
\newpage
Let $\eta_X(\epsilon)$ be the semi-flat cscK metric induced from
the semi-flat cscK metric $\eta(\epsilon)$ on $\mathbb{H}^+\times_{Id}\mathbb{C}$:
\begin{equation*}
    \eta(\epsilon)=\sqrt{-1}\partial\bar{\partial}\left(-\log v^2+\epsilon
    \frac{y^2}{v}\right)
\end{equation*}

then $\eta_X(\epsilon)=\Phi^*\eta_1+\epsilon \eta_2$
where $\eta_1=s^*\eta_X$ induced from $\sqrt{-1}\partial\bar{\partial}(-\log v^2)$, and for any section $s_i\in MW(X/\Sigma_g)$, we have
\begin{equation*}
\begin{split}
     [\Phi^*\eta_1]\cdot [s_i]= \int_{s_i(\Sigma_g^*)}\Phi^*\eta_1&=\int_{\Sigma_g^*}(\Phi\circ s_i)^*\eta_1=
    \int_{\Sigma_g^*}\eta_1=[\Phi^*\eta_1]\cdot [s]
\end{split}
\end{equation*}
and
\begin{equation*}
  [\Phi^*\eta_1]\cdot [F] =0
\end{equation*}
Therefore:
\begin{equation*}
\begin{split}
      [\eta_X(\epsilon)]\cdot \psi(s_i)&= [\eta_X(\epsilon)]\cdot
      ([\Phi^*\eta_1]+\epsilon \eta_2)\cdot ([s_i]-[s]-([s_i]\cdot [s]+\chi)[F])\\
     &=\epsilon  [\eta_X(\epsilon)]\cdot \psi(s_i)
\end{split}
\end{equation*}

\begin{corollary}
    The associated class of $[\eta_X(\epsilon)]$ is:
    \begin{equation*}
        [\eta_X(\epsilon)]=\frac{\pi d}{3}[F]+\epsilon\left(
        [s]+\chi[F]+\sum_{p}N_p+E_X\right)
    \end{equation*}
\end{corollary}
Let
\begin{equation}\label{DX}
    D_X:=[s]+\sum_{p}N_p+E_X
\end{equation}

\begin{theorem}\label{series}
If the genus $g\ge 1$, An $H(X,\Phi,s)$-invariant semi-flat cscK current $\eta(t)$ in the class $K_X+tD_X$ exists if and only if  $t\in \left(0,1+\frac{2g-2}{\chi}\right)$, here $d$
    is the degree of the Jacobian $J:\Sigma_g \rightarrow \mathbb{P}^1$,
    and $\chi$ is the Euler-Poinca\'re character of $X$.Furthermore, if $X$
possesses at least one singular fiber $X_p$ other than of type $I_b,I_b^*$,
then this current is unique. The fiber volume and scalar curvarure of $\eta(t)$ are given by:
\begin{equation*}
\begin{split}
    \text{fiber volume} &=t\\
\text{scalar curvature}&=\frac{-2\pi d}{2g-2+\chi(1-t)}
\end{split}
\end{equation*}
\end{theorem}

\begin{proof}
    By the Kodaira canonical bundle formua:
    \begin{equation*}
        K_X=(2g-2+\chi)[F]
    \end{equation*}
therefore:
\begin{equation*}
    \frac{2g-2+\chi}{\pi d/3+\epsilon\chi} [\eta_X(\epsilon)]
    =K_X+\frac{\epsilon(2g-2+\chi)}{\pi d/3+\epsilon\chi}D_X
\end{equation*}
notice that:
\begin{equation*}
    0<\frac{\epsilon(2g-2+\chi)}{\pi d/3+\epsilon\chi}<1+\frac{2g-2}{\chi},\forall \epsilon>0
\end{equation*}
take $t=\frac{\epsilon(2g-2+\chi)}{\pi d/3+\epsilon\chi}$
and $\eta(t)=\frac{t}{\epsilon}\eta_X(\epsilon)$ is desired.
\end{proof}

\section{Further questions}
There are some further questions about this topic:

\textbf{Question 1:} If $X$ possesses singular fibers only of type $I_b,I_b^*$,
does the uniqueness for this current still hold?

\textbf{Question 2:} How can we characterize the essential part $E_X\in T(X)^{\perp}$ of the associated class of this current ?

Furthermore, we suggest a certain uniformization for the elliptic surface:
\begin{conjecture}
Let $\Phi:X\rightarrow \Sigma_g$ be a minimal elliptic surface,
    there is a unique semi-flat cscK current on $X$ in the sense that if $\omega_1,\omega_2$ are two semi-flat cscK current (with fiber volume $1$ and scalar cuvature $-3$) then there exists $(h',h)\in H(X,\Phi)$, such that:
    \begin{equation*}
        h'_*\omega_1=\omega_2
    \end{equation*}
\end{conjecture}
There are some facts about semi-flat cscK currents:

\textbf{Fact 1:} If $\omega$ is a semi-flat cscK current on $X$, then for any $(h',h)\in H(X,\Phi)$ , $h'_*\omega$ is again a semi-flat cscK current with same fiber volume and scalar curvature.

\textbf{Fact 2:} If $\Phi:X\rightarrow \Sigma_g$ is a minimal elliptic surface possessing multiple fibers, then there exists a minimal elliptic surface $\Phi: X'\rightarrow \Sigma_{g'}$ with section, and a finite branched covering pair $(L',L)$ commutes following diagram \cite{Ko63}:
\begin{equation*}
    \xymatrix{
    X'\ar[r]^{L'}\ar[d]_{\Phi'} & X\ar[d]^{\Phi} \\
    \Sigma_{g'}\ar[r]^{L} & \Sigma_g
    }
\end{equation*}

Consequently, if $\omega'$ is a semi-flat cscK current on $X'$, then $L'_*\omega'$
defines a semi-flat cscK current on $X$.

\bibliographystyle{abbrv}
\bibliography{reference}

\begin{thebibliography}{10}

\bibitem{Aubin76}
T.~Aubin.
\newblock Equations du type {Monge-Amp{\`e}re} sur les vari{\'e}t{\'e}s
  k{\"a}hl{\'e}riennes compactes.
\newblock {\em Comptes Rendus de l'Acad{\'e}mie des Sciences, Paris, S{\'e}rie
  A}, 283(3):119--121, 1976.

\bibitem{BeRo82}
R.~Berndt.
\newblock Sur l’arithm{\'e}tique du corps des fonctions elliptiques de niveau
  {N}.
\newblock In {\em Seminar on Number Theory, Paris}, volume~83, pages 21--32,
  1982.

\bibitem{Cao86}
H.-D. Cao.
\newblock Deformation of kahler metrics to kahler-einstein metrics on compact
  kahler manifolds.
\newblock Technical report, Princeton Univ., NJ (USA), 1986.

\bibitem{CaoGuenanciaPaun23}
J.~Cao, H.~Guenancia, and M.~P{\u{a}}un.
\newblock Variation of singular k{\"a}hler-einstein metrics: Kodaira dimension
  zero.
\newblock {\em Journal of the European Mathematical Society (EMS Publishing)},
  25(2), 2023.

\bibitem{Dy22}
Z.~S. Dyrefelt.
\newblock Existence of csck metrics on smooth minimal models.
\newblock {\em ANNALI SCUOLA NORMALE SUPERIORE-CLASSE DI SCIENZE}, pages
  223--232, 2022.

\bibitem{EyViZe09}
P.~Eyssidieux, V.~Guedj, and A.~Zeriahi.
\newblock Singular k{\"a}hler-einstein metrics.
\newblock {\em Journal of the American Mathematical Society}, 22(3):607--639,
  2009.

\bibitem{Fine04}
J.~Fine.
\newblock Constant scalar curvature k{\"a}hler metrics on fibred complex
  surfaces.
\newblock {\em Journal of Differential Geometry}, 68(3):397--432, 2004.

\bibitem{Fong-FT-ZY20}
F.~T.-H. Fong and Y.~Zhang.
\newblock Local curvature estimates of long-time solutions to the
  k{\"a}hler-ricci flow.
\newblock {\em Advances in Mathematics}, 375:107416, 2020.

\bibitem{Fong-FT-ZZ15}
F.~T.-H. Fong and Z.~Zhang.
\newblock The collapsing rate of the k{\"a}hler--ricci flow with regular
  infinite time singularity.
\newblock {\em Journal f{\"u}r die reine und angewandte Mathematik (Crelles
  Journal)}, 2015(703):95--113, 2015.

\bibitem{Gross-Tosatti-Zhang20}
M.~Gross, V.~Tosatti, and Y.~Zhang.
\newblock Geometry of twisted k{\"a}hler--einstein metrics and collapsing.
\newblock {\em Communications in Mathematical Physics}, 380(3):1401--1438,
  2020.

\bibitem{GrWi20}
M.~Gross and P.~M. Wilson.
\newblock Large complex structure limits of {K3} surfaces.
\newblock {\em Journal of Differential Geometry}, 55(3):475--546, 2000.

\bibitem{GuoSongWe16}
B.~Guo, J.~Song, and B.~Weinkove.
\newblock Geometric convergence of the k{\"a}hler--ricci flow on complex
  surfaces of general type.
\newblock {\em International Mathematics Research Notices},
  2016(18):5652--5669, 2016.

\bibitem{GuoJianShiSong24}
B.~Guo, J.~Song, and B.~Weinkove.
\newblock Geometric convergence of the k{\"a}hler--ricci flow on complex
  surfaces of general type.
\newblock {\em International Mathematics Research Notices},
  2016(18):5652--5669, 2016.

\bibitem{Hamilton82}
R.~S. Hamilton.
\newblock Three-manifolds with positive ricci curvature.
\newblock {\em Journal of Differential geometry}, 17(2):255--306, 1982.

\bibitem{HaPo75}
R.~Harvey and J.~Polking.
\newblock Extending analytic objects.
\newblock {\em Communications on Pure and Applied Mathematics}, 28(6):701--727,
  1975.

\bibitem{JianShiSong18}
W.~Jian, Y.~Shi, and J.~Song.
\newblock A remark on constant scalar curvature k{\"a}hler metrics on minimal
  models.
\newblock {\em Proceedings of the American Mathematical Society},
  147(8):3507--3513, 2019.

\bibitem{JianSong22}
W.~Jian and J.~Song.
\newblock Diameter estimates for long-time solutions of the k{\"a}hler--ricci
  flow.
\newblock {\em Geometric and Functional Analysis}, 32(6):1335--1356, 2022.

\bibitem{Ko85}
R.~Kobayashi.
\newblock Einstein-k{\"a}hler v-metrics on open satake v-surfaces with isolated
  quotient singularities.
\newblock {\em Mathematische Annalen}, 272:385--398, 1985.

\bibitem{Ko63}
K.~Kodaira.
\newblock On compact analytic surfaces: {II}.
\newblock {\em Annals of Mathematics}, 77(3):563--626, 1963.

\bibitem{Liu20}
W.~Liu.
\newblock Convergence of csck metrics on smooth minimal models of general type.
\newblock {\em arXiv preprint arXiv:2012.09934}, 2020.

\bibitem{Shioda90}
T.~Shioda.
\newblock On the {Mordell-Weil} lattices.
\newblock {\em Commentarii Mathematici Universitatis Sancti Pauli},
  39(2):211--240, 1990.

\bibitem{Dyrefelt20}
Z.~Sj{\"o}str{\"o}m~Dyrefelt.
\newblock On k-polystability of csck manifolds with transcendental cohomology
  class.
\newblock {\em International Mathematics Research Notices}, 2020(9):2769--2817,
  2020.

\bibitem{Song20}
J.~Song.
\newblock Nakai-moishezon criterions for complex hessian equations.
\newblock {\em arXiv preprint arXiv:2012.07956}, 2020.

\bibitem{SongTian06}
J.~Song and G.~Tian.
\newblock The {K\"ahler-Ricci} flow on surfaces of positive {Kodaira}
  dimension.
\newblock {\em arXiv e-prints}, 2006.

\bibitem{SoTi12}
J.~Song and G.~Tian.
\newblock Canonical measures and k{\"a}hler-ricci flow.
\newblock {\em Journal of the American Mathematical Society}, 25(2):303--353,
  2012.

\bibitem{SongTian12}
J.~Song and G.~Tian.
\newblock Canonical measures and k{\"a}hler-ricci flow.
\newblock {\em Journal of the American Mathematical Society}, 25(2):303--353,
  2012.

\bibitem{SongTian16}
J.~Song and G.~Tian.
\newblock Bounding scalar curvature for global solutions of the
  k{\"a}hler-ricci flow.
\newblock {\em American Journal of Mathematics}, 138(3):683--695, 2016.

\bibitem{Song-Tian17}
J.~Song and G.~Tian.
\newblock The k{\"a}hler--ricci flow through singularities.
\newblock {\em Inventiones mathematicae}, 207(2):519--595, 2017.

\bibitem{SongTianZhang19}
J.~Song, G.~Tian, and Z.~Zhang.
\newblock Collapsing behavior of ricci-flat kahler metrics and long time
  solutions of the kahler-ricci flow.
\newblock {\em arXiv preprint arXiv:1904.08345}, 2019.

\bibitem{SoWe11}
J.~Song and B.~Weinkove.
\newblock The k{\"a}hler--ricci flow on hirzebruch surfaces.
\newblock 2011.

\bibitem{TiZh06}
G.~Tian and Z.~Zhang.
\newblock On the k{\"a}hler-ricci flow on projective manifolds of general type.
\newblock {\em Chinese Annals of Mathematics, Series B}, 27:179--192, 2006.

\bibitem{Tian-Zhang07}
G.~Tian and Z.~Zhang.
\newblock Regularity of the k{\"a}hler--ricci flow.
\newblock {\em Comptes Rendus Mathematique}, 351(15-16):635--638, 2013.

\bibitem{TianZLZhang16}
G.~Tian and Z.~Zhang.
\newblock Convergence of k{\"a}hler--ricci flow on lower-dimensional algebraic
  manifolds of general type.
\newblock {\em International Mathematics Research Notices},
  2016(21):6493--6511, 2016.

\bibitem{TianZLZhang21}
G.~Tian and Z.~Zhang.
\newblock Relative volume comparison of ricci flow.
\newblock {\em Science China Mathematics}, 64:1937--1950, 2021.

\bibitem{Tosatti10}
V.~Tosatti.
\newblock K{\"a}hler-ricci flow on stable fano manifolds.
\newblock 2010.

\bibitem{TWY18}
V.~Tosatti, B.~Weinkove, and X.~Yang.
\newblock The k{\"a}hler-ricci flow, ricci-flat metrics and collapsing limits.
\newblock {\em American Journal of Mathematics}, 140(3):653--698, 2018.

\bibitem{Tsuji88}
H.~Tsuji.
\newblock Existence and degeneration of k{\"a}hler-einstein metrics on minimal
  algebraic varieties of general type.
\newblock {\em Mathematische Annalen}, 281:123--133, 1988.

\bibitem{Wang17}
B.~Wang.
\newblock The local entropy along ricci flow---part a: the no-local-collapsing
  theorems.
\newblock {\em arXiv preprint arXiv:1706.08485}, 2017.

\bibitem{Yang07}
J.-H. Yang.
\newblock Invariant metrics and {Laplacians} on {Siegel--Jacobi} space.
\newblock {\em Journal of Number Theory}, 127(1):83--102, 2007.

\bibitem{Yau78}
S.-T. Yau.
\newblock On the {Ricci} curvature of a compact {K{\"a}hler} manifold and the
  complex {Monge-Amp{\'e}re} equation, {I}.
\newblock {\em Communications on Pure and Applied Mathematics}, 31(3):339--411,
  1978.

\bibitem{Zagier08}
D.~Zagier.
\newblock Elliptic modular forms and their applications.
\newblock In {\em The 1-2-3 of modular forms: Lectures at a summer school in
  Nordfjordeid, Norway}, pages 1--103. Springer, 2008.

\bibitem{Zh25}
K.~Zhang.
\newblock The ricci iteration towards csck metrics.
\newblock {\em Advances in Mathematics}, 475:110340, 2025.

\bibitem{ZhZh19}
Y.~Zhang and Z.~Zhang.
\newblock The continuity method on minimal elliptic k{\"a}hler surfaces.
\newblock {\em International Mathematics Research Notices},
  2019(10):3186--3213, 2019.

\bibitem{Z.Zh09}
Z.~Zhang.
\newblock Scalar curvature bound for k{\"a}hler--ricci flows over minimal
  manifolds of general type.
\newblock {\em International Mathematics Research Notices},
  2009(20):3901--3912, 2009.

\bibitem{Z.Zh10}
Z.~Zhang.
\newblock Scalar curvature behavior for finite-time singularity of
  k{\"a}hler-ricci flow.
\newblock {\em Michigan mathematical journal}, 59(2):419--433, 2010.

\end{thebibliography}
\end{document}